\documentclass[10pt,a4paper]{article}

\usepackage{geometry}
\usepackage[utf8]{inputenc}
\usepackage{amsthm, amssymb, graphicx}
\usepackage[numbers,square,sort&compress]{natbib}
\usepackage{listings}
\usepackage[ruled,vlined]{algorithm2e}
\usepackage{amsmath,tikz}
\usepackage{mathtools}
\usepackage{cases}
\usepackage{algorithmic}
\usepackage{array, makecell}
\usepackage{sidecap, caption}
\usepackage{xcolor, color}
\usepackage{cleveref}
\setcitestyle{numbers, comma,\square}
\usepackage{bbm}
\usepackage{enumitem}

\newtheorem{theorem}{Theorem}[section]
\newtheorem{corollary}{Corollary}[section]
\newtheorem{lemma}{Lemma}[section]
\newtheorem{example}{Example}[section]
\newtheorem{remark}{Remark}[section]
\newtheorem{assumption}{Assumption}[section]
\newtheorem{definition}{Definition}[section]
\DeclareMathOperator*{\argmin}{argmin}

\newcommand{\R}{\mathbb{R}}
\def\g{{\tilde{g}}}

\newcommand{\ignore}[1]{}
\def\T{{\cal T}}
\def\C{{\cal C}}

\def\M{{\cal{M}}}

\def\s{{\bf s}}
\def\x{{\bf x}}
\def\y{{\bf y}}
\def\z{{\bf z}}
\def\w{{\bf w}}

\def\u{{\bf u}}
\def\v{{\bf v}}
\def\t{{\bf t}}
\def\0{{\bf 0}}
\def\A{{\cal A}}
\def\B{{\cal B}}
\def\C{{\cal C}}
\def\c{{\mathbf{c}}}

\def\T{{\cal T}}
\def\M{{\cal M}}

\def\H{{\cal H}}

\title{A Homogeneous Tensor Framework for High-Order Trust-Region and Spherical Polynomial Optimization}

\author{Wenqi Zhu\thanks{Mathematical Institute, University of Oxford, Woodstock Road, Oxford, UK, OX2 6GG. {\tt wenqi.zhu@maths.ox.ac.uk}}
\and
Haibin Chen\thanks{School of Management Science, Qufu Normal University, Rizhao, Shandong, China, 276800. {\tt {chenhaibin508@qfnu.edu.cn}}}
\and
Guanglu Zhou\thanks{School of Electrical Engineering, Computing and Mathematical Sciences, Curtin University, Perth, WA, Australia. {\tt g.zhou@curtin.edu.au}}}

\date{\today}

\begin{document}

\maketitle
\begin{abstract}

High-order methods can improve worst-case evaluation complexity, but for orders $p\geq3$ their Taylor subproblems are nonconvex polynomial optimization problems and are generally difficult to solve. We develop a radius-controlled boundary approach based on homogeneous tensor representations. By augmenting the step with a constant coordinate, any $p$th-order Taylor polynomial can be represented exactly as an order-$p$ homogeneous tensor form; at a prescribed radius, the boundary model is a spherical polynomial optimization problem. The representation applies to arbitrary $p$, while the algorithmic development focuses on the cubic case $p=3$. For an inhomogeneous cubic on the sphere, we introduce a quadratic shift and prove, under an explicit shift bound, equivalence with a three-block multilinear formulation at global optimality. This motivates a proximal alternating minimization (PAM) method with closed-form block updates; its objective values decrease and every accumulation point is stationary. We embed the boundary-step mechanism in an Adaptive Homogeneous Tensor Method (Ada--HTM). Under explicit smoothness, safeguarded-decrease, weak-curvature nondegeneracy, and local-refinement conditions, Ada--HTM attains the adaptive-regularization-type (AR$p$-type) evaluation complexity $\mathcal{O}(\epsilon^{-(p+1)/p})$ for first-order stationarity. Numerically, PAM matches order-$2$ moment--sum-of-squares (SOS) certificates on the structured cubic instances for which certification is tractable, scales particularly well for low-rank tensors, and makes Ada--HTM competitive with trust-region and cubic-regularization methods, with its largest gains on ill-conditioned and badly-scaled problems.
\end{abstract}



\section{Introduction}

High-order models offer a well-established route to improved worst-case complexity in smooth nonconvex optimization. If the $p$th derivative of $f:\mathbb R^n\to\mathbb R$ is Lipschitz continuous, adaptive regularization methods of order $p$ (AR$p$) require at most
\[
O\!\left(\epsilon^{-\frac{p+1}{p}}\right)
\]
evaluations of $f$ and its derivatives to compute a point satisfying $\|\nabla f\|\leq\epsilon$; this order is sharp for methods using derivatives through order $p$~\cite{birgin2017worst,cartis2020sharp,cartis2020concise,carmon2020lower,carmon2021lower}. Their computational bottleneck is the step calculation. At an iterate $\x_k$, the available derivatives define
\begin{equation}
T_{p}(\x_k, \s) := f(\x_k) + \sum_{j=1}^p \frac{1}{j!} \nabla_{\x}^j f(\x_k)[\s]^j,
\label{Taylor}
\end{equation}
where $\nabla_{\x}^j f(\x_k) \in \R^{n^j}$ denotes the $j$th derivative tensor of $f$ at $\x_k$ and $\nabla_{\x}^j f(\x_k)[\s]^j$ its $j$th-order action along $\s \in \R^n$. For $p\geq3$, this polynomial is generally nonconvex and unbounded below. Both the {trust-region (TR)} model
\begin{equation}
\min_{\|\s\|\leq\Delta_k}T_p(\x_k,\s)
\label{subprob TR}
\end{equation}
and the regularized model
\begin{equation}
\min_{\s\in\mathbb R^n}\left\{T_p(\x_k,\s)
+\frac{\sigma_k}{p+1}\|\s\|^{p+1}\right\},
\label{subprob}
\end{equation}
where $\Delta_k>0$ denotes the {TR} radius and $\sigma_k>0$ the regularization parameter, therefore lead to nonconvex polynomial subproblems. Cubic optimization on a sphere is already NP-hard in the worst case~\cite{nesterov2003random,luo2010semidefinite}; a practical high-order method must consequently distinguish structural reformulations, local inner-solver guarantees, and global optimality certificates.

This paper develops a homogeneous-tensor route to these subproblems. Rather than solve the interior problem~\eqref{subprob TR} directly, we use the Taylor model on the boundary of an adaptively selected radius. With $\tilde\s=(1,\s^\top)^\top$, there is a symmetric order-$p$ tensor $\A_k$ such that
\[
T_p(\x_k,\s)=\A_k[\tilde\s]^p.
\]
Thus, at a prescribed radius, the boundary step is a spherical polynomial problem. The identity is exact for every order $p$; the boundary restriction is an algorithmic design choice controlled by the radius schedule. The formulation is most useful when contractions with $\A_k$ are evaluated directly from derivatives, without assembling the augmented tensor, and it preserves sparse, low-rank, and other derivative structure.

The computational development concentrates on $p=3$. The resulting cubic spherical polynomial is inhomogeneous because it contains linear and quadratic terms. Existing multilinear equivalences for homogeneous tensor polynomials therefore do not apply directly~\cite{zhang2012cubic}. We subtract a quadratic shift, which is constant on the sphere, and prove that under an explicit shift bound the cubic problem and a three-block spherical multilinear problem have the same global value; every global multilinear optimizer recovers a global optimizer of the cubic problem. The multilinear structure yields a proximal alternating minimization (PAM) method with closed-form block updates. PAM is a practical nonconvex solver: its accumulation points are stationary, but the equivalence theorem does not turn PAM into a worst-case global solver.

We place this boundary solver inside the Adaptive Homogeneous Tensor Method (Ada--HTM). The analysis separates a far-field phase from near-stationary negative-, weak-, and positive-curvature regimes. In the far field, accepted steps satisfy an AR$p$-order decrease test. Near stationarity, curvature-dependent radii exploit a quadratic or cubic direction, or enter a locally convex refinement phase {costing $O(\log\log(1/\epsilon))$ further evaluations}. Under the stated smoothness and bounded-derivative assumptions, a weak-curvature nondegeneracy condition, and a local-refinement condition, the far-field cost dominates and gives $O(\epsilon^{-(p+1)/p})$ evaluations. This is a conditional AR$p$-type result, not an unconditional improvement for arbitrary high-order {TR} methods: a naive radius rule for a $p\geq3$ Taylor model may retain the $O(\epsilon^{-2})$ rate~\cite[Sec.~4.1.3]{cartis2022evaluation}. The outer analysis requires only comparison with an explicit gradient- or curvature-based reference step. Accordingly, PAM supplies a candidate and a deterministic safeguard supplies the decrease needed by the theory.

The contributions are as follows.
\begin{itemize}[topsep=2pt,itemsep=2pt,parsep=0pt,leftmargin=*]
\item \emph{Homogeneous representation.} We give an exact order-$p$ representation of an arbitrary $p$th-order Taylor polynomial and formulate radius-controlled boundary steps as spherical polynomial problems.

\item \emph{Cubic equivalence and solver.} For $p=3$, we establish global equivalence between the shifted cubic spherical problem and its three-block multilinear formulation. For PAM we prove descent, square-summability of the increments, and stationarity of accumulation points.

\item \emph{Conditional outer complexity.} We formulate Ada--HTM using a reference-safeguarded boundary step and, under explicit far-field and near-stationary assumptions, obtain the AR$p$ evaluation bound $O(\epsilon^{-(p+1)/p})$.

\item \emph{Two-level numerical evidence.} PAM matches order-$2$ moment--SOS lower bounds on all tested structured cubic instances for which certification is tractable and scales especially well for low-rank tensors. At the outer level, Ada--HTM is robust on Chebyshev--Rosenbrock through $n=10$ and shows its largest gains on ill-conditioned and badly-scaled Mor\'e--Garbow--Hillstrom (MGH) problems. {The two scales are complementary: PAM is a subproblem solver whose dimension is limited only by the contraction cost, whereas the outer robustness study is limited by the dimensions at which the benchmark's near-global outcome can be reliably decided.}
\end{itemize}

Our work connects high-order regularization~\cite{birgin2017worst,cartis2020concise,cartis2019universal}, structured polynomial-model solvers~\cite{cartis2023second,zhu2023cubic,zhu2022quartic,Nesterov2022quartic}, and spherical tensor optimization~\cite{absil2008optimization,chen2022efficient,chen2025tensor,kolda2014adaptive,zhang2012cubic}. Moment--SOS relaxations provide global lower bounds~\cite{lasserre2001global,ahmadi2023higher}, but their semidefinite representations grow rapidly; here they are used only as a posteriori certificates where tractable. {This paper provides the algorithmic and computational component of a broader program on tractable high-order Taylor subproblems; the SOS-based exactness and certification theory is developed separately~\cite{zhu2024global,zhu2026sufficiently,cai2026globally}.}

{
\paragraph{Related work: TR complexity.}
Classical TR methods based on quadratic models require
\(\mathcal{O}(\epsilon^{-2})\) evaluations for first-order stationarity under
standard radius updates \cite{conn2000trust,cartis2022evaluation}. TRACE
\cite{curtis2017trace} improves this to the \(\mathcal{O}(\epsilon^{-3/2})\)
rate of cubic regularization, which is optimal for second-order methods: its
contraction--expansion mechanism controls the TR multiplier \(\lambda_k\)
relative to the step norm, so that \(\lambda_k/\|\s_k\|\) acts as an implicit
cubic-regularization parameter. Its extension iTRACE~\cite{curtis2023itrace}
permits inexact subproblem solutions, preserving the optimal rate while
substantially reducing the cost in Hessian-vector products. For arbitrary
order, Cartis, Gould, and Toint~\cite{cartis2021inexactTR} analyze a TR
framework with high-order Taylor decrements and dynamically accurate
evaluations that targets strong \(q\)th-order approximate minimizers; its
analysis assumes steps attaining a fixed fraction of the globally optimal
Taylor decrement, an operation that remains computationally challenging
beyond \(p=2\). Ada--HTM addresses the complementary question: an
explicit-radius, boundary-constrained \(p\)th-order TR method that attains
the AR\(p\) first-order rate under the stated conditions, without a
global-decrement oracle, together with a concrete spherical polynomial
solver for the cubic subproblem.
}

\paragraph{Organization of the Paper}
\Cref{sec 2 ada htm} gives the homogeneous representation; \Cref{sec:CSPOP} develops the cubic formulation and the PAM solver; \Cref{sec: convergence} analyzes Ada--HTM; \Cref{sec numerical} reports numerical results; and \Cref{sec conclusion} concludes the paper.

\paragraph{Preliminaries}
Let $\mathbb{R}^n$ be the $n$-dimensional real Euclidean space. Generally, we use lowercase letters, boldface letters, and calligraphic uppercase letters to denote scalars, vectors, and tensors respectively (e.g., the scalar $a$, the decision variable $\x$, and the tensor $\A$).
The all-zero {vector} in $\mathbb{R}^n$ is denoted by $\0$. For any $\x,\y\in\mathbb{R}^n$, we use $\langle \x, \y\rangle$ to represent the standard inner product between $\x$ and $\y$, and $\|\x\|=\sqrt{\langle \x,\x\rangle}$ denotes the Euclidean norm. The superscript $^\top$ stands for the transpose of vectors and matrices.
{All third-order tensors in this paper are symmetric. For two symmetric tensors $\A, \B \in \mathbb{R}^{n\times n\times n}$, their inner product is defined as
$
\langle\A,~ \B\rangle=\sum_{i,j,l\in[n]}\A_{ijl}\B_{ijl},
$
where $[n]:=\{1,\dots,n\}$, and the induced Frobenius norm is $\|\A\|_F=\sqrt{\langle\A,\A\rangle}$.
Throughout the paper, $k$ indexes the outer iterations of \Cref{HTM algo} and $\kappa$ the inner iterations of the subproblem solver; $\epsilon\in(0,1)$ denotes the outer accuracy level.}

\section{Homogeneous Tensor Framework for High-Order Optimization}
\label{sec 2 ada htm}

In this section we present a homogeneous tensor formulation of high-order Taylor models and describe the resulting optimization framework.
This formulation reveals a direct connection between high-order trust-region methods and spherical polynomial optimization, and leads to a tensor subproblem that must be solved at each iteration.

\paragraph{Homogeneous Tensor Representation}

The $p$th-order Taylor model \eqref{Taylor} can be equivalently reformulated as a homogeneous tensor optimization problem on the sphere.
Let
\[
T_p(\x_k,\s) := \sum_{\nu\in\mathbb{Z}^n_+,|\nu|\leq p} f_{\nu}\s^{\nu}\in \mathbb{R}_{p}[\s],
\]
where $\s = [s_1, \dotsc, s_n]^\top \in \mathbb{R}^n$ and {$f_{\nu}$ denotes the coefficient corresponding to the monomial
$\s^{\nu}=s_1^{\nu_1}\cdots s_n^{\nu_n}$ with $|\nu|=\nu_1+\cdots+\nu_n$ (the multi-index is written $\nu$ to avoid a clash with the shift parameter $\alpha$ of \Cref{sec:CSPOP})}. Introduce the augmented vector
\[
\tilde{\s}=(1,\s^\top)^\top\in\mathbb{R}^{n+1}.
\]
Define the symmetric tensor $\A$ of order $p$ and dimension $n+1$ such that
\begin{equation}
\label{formulation for pth order}
\A_{\pi(i_1i_2\cdots i_{p})}
=
\frac{(p-|\nu|)!\,\nu_1!\cdots\nu_n!}{p!}f_{\nu},
\qquad
i_1,\dots,i_p\in\{0,1,\dots,n\},
\end{equation}
where $\pi(i_1\cdots i_p)$ denotes permutations of the indices.
Then the Taylor model admits the homogeneous tensor representation
\[
T_p(\x_k,\s)=\A[\tilde{\s}]^p.
\]
{
The quantity $\binom{n+p}{p}$ counts the independent symmetric coefficient
classes of $\A$, equivalently the monomials of degree at most $p$; it is not,
in general, the number of nonzero entries in the expanded tensor array. Dense
derivatives may therefore produce a dense expanded tensor. In implementation,
we avoid forming $\A$ and evaluate the required contractions directly from the
derivative tensors, exploiting sparsity or low rank whenever it is present.
The representation applies to arbitrary $p\ge3$, whereas the shifted
multilinear solver developed in the next section specializes to $p=3$.
}

\subsection{Adaptive Homogeneous Tensor Method}

The homogeneous tensor representation naturally leads to a high-order optimization framework that we call the {Ada-HTM}. At each iteration $k$, the algorithm constructs the tensor model
{
$
T_p(\x_k,\s)=\A[\tilde \s]^p
$
}
and computes a step by approximately solving the spherical tensor subproblem $
\min_{\|\s\|=\Delta_k} \A[\tilde{\s}]^p .
$
The candidate step is evaluated using a ratio test comparing the actual reduction in the objective function with the reduction predicted by the Taylor model.
Based on this ratio, the step is accepted or rejected, and the trust-region radius $\Delta_k$ is updated accordingly.
Depending on the magnitude of the gradient and the curvature of the Hessian, the algorithm operates in different regimes. When the gradient is sufficiently large, the tensor subproblem is solved to obtain a descent step. When the gradient becomes small, the algorithm transitions to refinement steps and/or exploits local curvature information.

{
The practical inner solver and the outer analysis are connected by a simple
reference safeguard. Let $\widehat\s_k$ be any candidate returned by PAM (or
another boundary solver), with $\|\widehat\s_k\|=\Delta_k$. In the far field
set $\s_k^{(g)}=-\Delta_k\mathbf g_k/\|\mathbf g_k\|$; in a near-stationary
nonconvex regime let $\s_k^{(h)}$ be the signed smallest-eigenvalue direction
defined in \Cref{def:reference-direction}. We use
\begin{equation}
\s_k\in\argmin_{\s\in\{\widehat\s_k,\s_k^{\rm ref}\}}
T_p(\x_k,\s),
\qquad
\s_k^{\rm ref}=\s_k^{(g)}\ \hbox{or}\ \s_k^{(h)}.
\label{eq:safeguarded-boundary-step}
\end{equation}
Thus PAM proposes the computationally effective step, while the explicit
reference direction guarantees the model comparison used in the analysis.
No global solution of the boundary problem is required.
}

\begin{algorithm}[!htbp]
\caption{Adaptive $p$th-order Taylor Homogeneous Tensor Method (Ada-HTM)}
\label{HTM algo}
\begin{algorithmic}[1]
\STATE \textbf{Input:} accuracy level $0<\epsilon\ll1$; constants $\eta\in(0,1)$, $\gamma_2>1>\gamma_1>0$, $a:=\frac{p+1}{2p}$, $b:=\frac{p+1}{4p}$; initial point $\x_0\in\R^n$; initial radius $\Delta_0=\|\nabla_{\x} f(\x_0)\|>0$; set $k=0$.
\WHILE{$\|\nabla_{\x} f(\x_k)\|\ge\epsilon$}
\IF{$\|\nabla_{\x} f(\x_k)\|\ge\epsilon^a$ \textbf{(far-field regime)}}
\STATE Form $T_p(\x_k,\s)=\A[\tilde{\s}]^p$ as in \eqref{formulation for pth order}, generate a boundary candidate $\widehat\s_k$ with $\|\widehat\s_k\|=\Delta_k$ {using an order-$p$ spherical polynomial solver (for $p=3$, \Cref{alg1})}, and set
\begin{equation}\label{HTM step}
{\s_k\in\argmin_{\s\in\{\widehat\s_k,\,\s_k^{(g)}\}}T_p(\x_k,\s)}
\end{equation}
{with the safeguard \eqref{eq:safeguarded-boundary-step}}.
\STATE Compute $T_p(\x_k,\s_k)$, the ratio
\begin{equation}\tag{Ratio Test}\label{ratio test}
\rho_k:=\frac{f(\x_k)-f(\x_k+\s_k)}{f(\x_k)-T_p(\x_k,\s_k)},
\end{equation}
and check the model decrease condition
\begin{equation}\tag{Model Decrease Test}\label{Model Decrease test}
f(\x_k)-T_p(\x_k,\s_k)\ge\frac{\Delta_k\epsilon^a}{p+1}>0.
\end{equation}
\IF{$\rho_k>\eta$ \textbf{and} \eqref{Model Decrease test} holds}
\STATE Successful iteration: $\x_{k+1}:=\x_k+\s_k$, \quad $\Delta_{k+1}:=\gamma_2\Delta_k$.
\ELSE
\STATE Unsuccessful iteration: $\x_{k+1}:=\x_k$, \quad $\Delta_{k+1}:=\gamma_1\Delta_k$.
\ENDIF
\ELSIF{$\lambda_{\min}(\mathbf{H}_k)\le\epsilon^b$}
\STATE \textbf{Near-stationary nonconvex regime:} {set $\Delta_k=C_{\mathrm{ns}}\,\epsilon^{b}$ as in \eqref{eq:near-radius}; generate $\widehat\s_k$ with the same boundary solver and safeguard it against $\s_k^{(h)}$ using \eqref{eq:safeguarded-boundary-step}}; set $\x_{k+1}:=\x_k+\s_k$.
\ELSE
\STATE \textbf{Near-stationary convex regime:} compute an approximate minimizer $\s_k$ of $T_p(\x_k,\s)$ over the local convex region $\Omega_k$ of \Cref{sec Refinement} {and apply the local Newton refinement of \Cref{cor:newton-refine}}; set $\x_{k+1}:=\x_k+\s_k$.
\ENDIF
\STATE $k:=k+1$.
\ENDWHILE
\end{algorithmic}
\end{algorithm}

\paragraph{Link and comparison to AR$p$ and trust-region methods}
\Cref{HTM algo} mirrors the adaptive mechanism of AR$p$: increasing the radius $\Delta_k$ after a successful iteration ($\rho_k>\eta$) corresponds to decreasing the regularization and allowing a longer step, while decreasing $\Delta_k$ after an unsuccessful iteration corresponds to increasing the regularization and shortening the step. We prove in \Cref{sec: convergence} that the mechanism guarantees a uniform lower bound $\Delta_k\ge\Delta_{\min}>0$ in the far-field regime.
Compared with the AR$p$ algorithm, \Cref{HTM algo} requires one additional acceptance condition, the \eqref{Model Decrease test}. In AR$p$, the step minimizes the regularized model over $\R^n$,
$
\Check{\s}_k = \argmin_{\s \in \R^n} m_p(\x_k, \s),
$
so that
 {
 $$
   f(\x_k) - T_p(\x_k, \Check{\s}_k) = \underbrace{\big[f(\x_k) - m_p(\x_k, \Check{\s}_k) \big]}_{\ge 0} + \frac{\sigma_k}{p+1} \|\Check{\s}_k\|^{p+1} \ge \frac{\sigma_{\min}}{p+1} \|\Check{\s}_k\|^{p+1} \ge 0,
$$
}
and a model decrease of the required size is automatic. In \Cref{HTM algo}, the minimization is instead performed on the Taylor model itself under the boundary constraint $\|\s\| = \Delta_k$; {the model-decrease test and the reference safeguard \eqref{eq:safeguarded-boundary-step} are therefore needed to ensure that $f(\x_k) - T_p(\x_k, \s_k) > 0$ is sufficiently large.}

The efficiency of the proposed framework depends critically on solving the tensor subproblem efficiently.
In the next section we study the cubic case ($p=3$), where the subproblem reduces to a cubic spherical polynomial optimization problem and can be solved using a specialized algorithm.
The convergence and evaluation complexity of the overall framework are analyzed in \Cref{sec: convergence}.

\section{Optimization of the Cubic Polynomial on Spherical Constraints}
\label{sec:CSPOP}

The cubic case $p=3$ plays a central role in the proposed framework: the tensor
subproblem appearing in Ada--HTM reduces exactly to a cubic spherical polynomial
optimization problem,
\begin{equation}\label{e1 cspop}
\begin{aligned}
\min_{\mathbf{x} \in \mathbb{R}^n}
&\hspace{0.3cm}
T_3(\x)= \textbf{T} {\bf x}^3 + \mathbf{x}^\top \textbf{H} \mathbf{x} + \textbf{g}^{\top}\mathbf{x}\\
\textup{s.t.}~
&\hspace{0.3cm}\|\mathbf{x}\|=1,
\end{aligned}
\tag{CSPOP}
\end{equation}
where $\textbf{T} \in \mathbb{R}^{n\times n\times n}$ is a third-order symmetric tensor, $\textbf{H}\in\mathbb{R}^{n\times n}$ is a symmetric matrix, and $\textbf g\in\mathbb{R}^n$.
Since the analysis in this section applies to general cubic polynomials,
we drop the iteration index $k$ for notational simplicity.
Throughout this section, $\x$ denotes a generic variable on the sphere,
and should not be confused with the iterate $\x_k$ used in the main algorithm.
Problem \eqref{e1 cspop} can be viewed as the third-order trust-region subproblem restricted to the boundary of the unit ball. To facilitate algorithmic design, we further introduce a multilinear reformulation of the homogeneous polynomial.
Using \eqref{formulation for pth order}, problem \eqref{e1 cspop} can be written as
\begin{equation}\label{e2 H cspop}
\begin{aligned}
\min_{\mathbf{x} \in \mathbb{R}^n}
&\hspace{0.3cm}
T_3(\x)=\langle\T,\tilde{\x}\circ\tilde{\x}\circ\tilde{\x}\rangle\\
\textup{s.t.}~
&\hspace{0.3cm}\|\mathbf{x}\|=1,
\end{aligned}
\tag{H-CSPOP}
\end{equation}
where $\tilde{\x}=(1,\x^\top)^\top\in\mathbb{R}^{n+1}$, {and $\T$ is the coefficient tensor of $T_3(\x)$}. The detailed homogeneous formulation with explicit coefficients is given in \Cref{sec example}.
A natural relaxation is the cubic spherical multilinear polynomial optimization problem
\begin{equation}\label{e3 CSMOP}
\begin{aligned}
\min_{\x,\y,\z \in \mathbb{R}^n}
&\hspace{0.3cm}
F(\x,\y,\z)=\langle\T,\tilde{\x}\circ\tilde{\y}\circ\tilde{\z}\rangle\\
\textup{s.t.}~~~
&\hspace{0.3cm}\|\mathbf{x}\|=1,\ \|\mathbf{y}\|=1,\ \|\mathbf{z}\|=1,
\end{aligned}
\tag{CSMOP}
\end{equation}
where $\tilde{\y}=(1,\y^\top)^\top$ and $\tilde{\z}=(1,\z^\top)^\top$.  For homogeneous polynomials, Zhang, Ye, and Qi~\cite{zhang2012cubic} established an equivalence between formulations such as \eqref{e2 H cspop} and \eqref{e3 CSMOP}.
However, since $T_3$ is generally non-homogeneous, the associated Lagrange multipliers need not coincide and their approach cannot be directly applied. To overcome this difficulty we introduce a quadratic shift
{
\[
\tilde T_3(\x)=T_3(\x)-\alpha\|\x\|^2,
\]
}
where $\alpha>0$.
Problems \eqref{e1 cspop} and
\begin{equation}\label{e4 CSPOP alpha}
\begin{aligned}
\min_{\x \in \mathbb{R}^n}
&\hspace{0.5cm}
\tilde{T_3}(\x)=T_3(\x)-\alpha\|\x\|^{2}\\
\textup{s.t.}~
&\hspace{0.5cm}\|\x\|=1
\end{aligned}
\tag{CSPOP$_\alpha$}
\end{equation}
share the same minimizers on the sphere.
For sufficiently large $\alpha$, this shift introduces negative curvature in restricted directions while preserving optimal solutions (see {\Cref{lema1},\Cref{thm1}}). The key advantage of this transformation is that it converts the optimization problem \eqref{e1 cspop} into the problem of minimizing the quadratic shifted polynomial ${\tilde T_3(\x)}$. We prove that for sufficiently large $\alpha>0$, the minimizer of \eqref{e1 cspop} can be obtained by minimizing the corresponding multilinear polynomial optimization problem associated with {$\tilde T_3(\x)$} (\Cref{corol1} and \Cref{them2}). Using tensor augmentation, define a symmetric tensor $\bar{\T}\in\mathbb{R}^{(n+1)\times(n+1)\times(n+1)}$ whose nonzero entries are
\[
\bar{\T}_{ii0}=\bar{\T}_{i0i}=\bar{\T}_{0ii}=\frac13,
\qquad i\in[n],
\]
and zero otherwise. Then
$
\alpha\|\x\|^2
=
\alpha\langle\bar{\T},\tilde{\x}\circ\tilde{\x}\circ\tilde{\x}\rangle .
$
Hence \eqref{e4 CSPOP alpha} admits the homogeneous formulation
\begin{equation}\label{e5 H CSPOP alpha}
\begin{aligned}
\min_{\mathbf{x} \in \mathbb{R}^n}
&\hspace{0.3cm}
\langle\T,\tilde{\x}\circ\tilde{\x}\circ\tilde{\x}\rangle
-
\alpha\langle\bar{\T},\tilde{\x}\circ\tilde{\x}\circ\tilde{\x}\rangle\\
\textup{s.t.}~
&\hspace{0.3cm}\|\mathbf{x}\|=1.
\end{aligned}
\tag{H-CSPOP$_\alpha$}
\end{equation}
The corresponding multilinear formulation is
\begin{equation}\label{e6 H CSMOP alpha}
\begin{aligned}
\min_{\x,\y,\z \in \mathbb{R}^n}
&\hspace{0.3cm}
\tilde F(\x,\y,\z)
=
\langle\T,\tilde{\x}\circ\tilde{\y}\circ\tilde{\z}\rangle
-
\alpha\langle\bar{\T},\tilde{\x}\circ\tilde{\y}\circ\tilde{\z}\rangle\\
\textup{s.t.}~~~
&\hspace{0.3cm}
\|\mathbf{x}\|=1,\
\|\mathbf{y}\|=1,\
\|\mathbf{z}\|=1.
\end{aligned}
\tag{H-CSMOP$_\alpha$}
\end{equation}
Exploiting the resulting multi-block structure, in \Cref{sec pam algo} we develop a proximal alternating minimization (PAM) algorithm for solving \eqref{e1 cspop}.

Denote $D=\{\x\in\mathbb{R}^n \mid \|\x\|=1\}$.
We first analyze how the magnitude of the quadratic shift influences the second-order curvature of the shifted polynomial along the tangent space of $D$, as characterized in Lemma~\ref{lema1}. {With a slight abuse of notation, we write $\tilde{T_3}(\tilde{\w}):=\langle\T-\alpha\bar{\T},\tilde{\w}\circ\tilde{\w}\circ\tilde{\w}\rangle$ for the homogeneous form of the shifted polynomial evaluated at $\tilde{\w}\in\mathbb{R}^{n+1}$, so that $\tilde{T_3}(\x)=\tilde{T_3}(\tilde{\x})$ whenever $\tilde{\x}=(1,\x^\top)^\top$.}

\begin{lemma}\label{lema1}
Let $\alpha\geq 3\sqrt{2}\|\T\|_F$. For any $\tilde{\w}=(1,\w^\top)^\top\in\mathbb{R}^{n+1}$ with $\|\w\|=1$,
\begin{eqnarray}
\label{suff cond}
\tilde{\u}^\top\nabla^2\tilde{T_3}(\tilde{\w})\tilde{\u}\leq0,~\forall~ \tilde{\u}=(0,\u^\top)^\top\in\mathbb{R}^{n+1}.
\end{eqnarray}
\end{lemma}

\proof By the symmetry of $\T$ and $\bar{\T}$, for any $\tilde{\w}=(w_0,\ldots,w_n)^\top\in\mathbb{R}^{n+1}$,
$$
\nabla \tilde{T_3}(\tilde{\w})=3\T\tilde{\w}\tilde{\w}-3\alpha\bar{\T}\tilde{\w}\tilde{\w},
\qquad
\nabla^2 \tilde{T_3}(\tilde{\w})=6\T\tilde{\w}-6\alpha\bar{\T}\tilde{\w},
$$
where $\T\tilde{\w}\tilde{\w}\in\mathbb{R}^{n+1}$ and $\T\tilde{\w}\in\mathbb{R}^{(n+1)\times(n+1)}$ have entries
$(\T\tilde{\w}\tilde{\w})_i=\sum_{j,l=0}^n\T_{ijl}w_jw_l$ and $(\T\tilde{\w})_{ij}=\sum_{l=0}^n\T_{ijl}w_l$.
For $\tilde{\w}=(1,\w^\top)^\top\in\mathbb{R}^{n+1}, \|\w\|=1$, and any $\tilde{\u}=(0,\u^\top)^\top\in\mathbb{R}^{n+1}$, it holds that
$$
\begin{aligned}
\tilde{\u}^\top\nabla^2 \tilde{T_3}(\tilde{\w})\tilde{\u}&=6\langle\T,\tilde{\w}\circ\tilde{\u}\circ\tilde{\u}\rangle-6\alpha\langle\bar{\T},\tilde{\w}\circ\tilde{\u}\circ\tilde{\u}\rangle\\
&=6\langle\T,\tilde{\w}\circ\tilde{\u}\circ\tilde{\u}\rangle-2\alpha\|\u\|^2\leq(6\sqrt{2}\|\T\|_F-2\alpha)\|\u\|^2\leq 0,
\end{aligned}
$$
since $\alpha\geq 3\sqrt{2}\|\T\|_F$, and the desired result holds.
\qed

\begin{remark}

The sufficient condition \eqref{suff cond} in Lemma \ref{lema1} is weaker than the usual requirement that the objective function be concave when studying optimization problems for quartic homogeneous polynomials on the sphere \cite{chen2022efficient,chen2025tensor}. Note that concavity would require
$
\z^\top \nabla^2 \tilde{T_3}(\tilde \w) \z \le 0, \forall~ \z\in\mathbb R^{n+1},
$
that is, the Hessian must be negative semidefinite in all directions. However, in Lemma \ref{lema1}, we only require
\[
\tilde \u^\top \nabla^2 \tilde{T_3}(\tilde \w)\tilde \u\le0,~~ \forall~ \tilde \u=(0,\u^\top)^\top.
\]
Hence Lemma \ref{lema1} imposes negative curvature only along directions affecting $\w$, while allowing arbitrary curvature in directions involving the first coordinate.
Therefore the condition is strictly weaker than global concavity of $\tilde{T_3}$.
\end{remark}

We now prove the equivalence between \eqref{e5 H CSPOP alpha} and \eqref{e6 H CSMOP alpha} in the sense that they share the same optimal value.

\begin{theorem}\label{thm1}
Let $\alpha\geq 3\sqrt{2}\|\T\|_F$. Then the optimal values of \eqref{e5 H CSPOP alpha} and \eqref{e6 H CSMOP alpha} are equal.
\end{theorem}

\proof Let $\tilde{\x}=(1,\x^\top)^\top\in\mathbb{R}^{n+1}, \tilde{\y}=(1,\y^\top)^\top\in\mathbb{R}^{n+1}$, where $\x, \y\in D$.
By Lemma \ref{lema1}, we have that
$$
\begin{aligned}
0\geq&\langle\T-\alpha\bar{\T},\tilde{\x}\circ(\tilde{\x}-\tilde{\y})\circ(\tilde{\x}-\tilde{\y})\rangle+2\langle\T-\alpha\bar{\T},\tilde{\y}\circ(\tilde{\x}-\tilde{\y})\circ(\tilde{\x}-\tilde{\y})\rangle\\
=&\langle\T-\alpha\bar{\T},\tilde{\x}\circ\tilde{\x}\circ\tilde{\x}\rangle+2\langle\T-\alpha\bar{\T},\tilde{\y}\circ\tilde{\y}\circ\tilde{\y}\rangle-3\langle\T-\alpha\bar{\T},\tilde{\x}\circ\tilde{\y}\circ\tilde{\y}\rangle\\
=&\tilde{T_3}(\x)+2\tilde{T_3}(\y)-3\tilde{F}(\x,\y,\y),\\
\end{aligned}
$$
which implies that
$$
3\min \{\tilde{T_3}(\x)~|~\x\in\mathbb{R}^n, \|\x\|=1\}\leq \tilde{T_3}(\x)+2\tilde{T_3}(\y)\leq 3\tilde{F}(\x,\y,\y).
$$
Since $\x,\y\in D$ are arbitrary, it follows that
\begin{equation}\label{e8}
\min \{\tilde{T_3}(\x)~|~\x\in D\}\leq\min\{\tilde{F}(\x,\y,\y)~|~\x, \y\in D\}.
\end{equation}

On the other hand, by Lemma \ref{lema1} again, for any $\tilde{\z}=(1, \z^\top)^\top, \tilde{\x}=(1, \x^\top)^\top, \tilde{\y}=(1, \y^\top)^\top$, $\x, \y, \z\in D$,
it holds that
$$
\begin{aligned}
0&\geq\langle\T-\alpha\bar{\T}, \tilde{\z}\circ(\tilde{\x}-\tilde{\y})\circ(\tilde{\x}-\tilde{\y})\rangle\\
&=\langle\T-\alpha\bar{\T}, \tilde{\z}\circ\tilde{\x}\circ\tilde{\x}\rangle+\langle\T-\alpha\bar{\T}, \tilde{\z}\circ\tilde{\y}\circ\tilde{\y}\rangle-2\langle\T-\alpha\bar{\T}, \tilde{\z}\circ\tilde{\x}\circ\tilde{\y}\rangle\\
&=\tilde{F}(\z,\x,\x)+\tilde{F}(\z,\y,\y)-2\tilde{F}(\z,\x,\y),
\end{aligned}
$$
i.e.,
\begin{equation}\label{e9}
\min\{\tilde{F}(\x,\y,\z)~|~\x, \y, \z\in D\}\geq \min\{\tilde{F}(\x,\y,\y)~|~\x, \y\in D\},
\end{equation}
where $\tilde{F}(\z,\x,\y)=\tilde{F}(\x,\y,\z)$ since the symmetry of $\T$ and $\bar{\T}$. Combining \eqref{e8}-\eqref{e9} with the apparent result below:
$$
\min\{\tilde{F}(\x,\y,\z)~|~\x, \y, \z\in D\}\leq \min\{\tilde{F}(\x,\y,\y)~|~\x, \y\in D\}\leq \min \{\tilde{T_3}(\x)~|~\x\in D\},
$$
we know that the optimal values of \eqref{e5 H CSPOP alpha} and \eqref{e6 H CSMOP alpha} are equal, and the desired results hold.
\qed

\begin{remark}
Denote the optimal value of \eqref{e1 cspop} (or its homogeneous formulation \eqref{e2 H cspop}) by $T_3^*$,
the optimal value of \eqref{e4 CSPOP alpha} (or equivalently its homogeneous formulation \eqref{e5 H CSPOP alpha}) by $\tilde{T}_3^*$,
and the optimal value of \eqref{e6 H CSMOP alpha} by $\tilde{F}^*$.
From \Cref{thm1}, when $\alpha$ is sufficiently large, problems \eqref{e5 H CSPOP alpha} and \eqref{e6 H CSMOP alpha} are equivalent in the sense that
\[
\tilde{T}_3^*=\tilde{F}^*.
\]
Moreover, since the quadratic shift does not change the minimizers on the sphere $D=\{\x\in\mathbb{R}^n\mid \|\x\|=1\}$, we also have
\[
\arg\min\{T_3(\x)\mid \x\in D\}
=
\arg\min\{\tilde{T}_3(\x)\mid \x\in D\}.
\]
Consequently, the optimal value of the original \eqref{e1 cspop} can be recovered from the solution of \eqref{e6 H CSMOP alpha} via
\[
T_3^*=\tilde{F}^*+\alpha .
\]
\end{remark}
Without the sufficient condition of Lemma~\ref{lema1}, \eqref{e5 H CSPOP alpha} and \eqref{e6 H CSMOP alpha} may fail to share the same optimal value, as the following counterexample shows.

\begin{example} \label{ex 1}
{Let $\T\in\mathbb{R}^{2\times 2\times 2}$ be symmetric with $\T_{122}=\T_{212}=\T_{221}=1$ and all other entries zero, and take $\tilde{\x}=(1,x)^\top$, $\tilde{\y}=(1,y)^\top$, $\tilde{\z}=(1,z)^\top$ with $x,y,z\in\{-1,1\}$; the condition \eqref{suff cond} of Lemma~\ref{lema1} fails for this $\T$. Then $\langle\T, \tilde{\x}\circ\tilde{\x}\circ\tilde{\x}\rangle=3x^2$ and $\langle\T,\tilde{\x}\circ\tilde{\y}\circ\tilde{\z}\rangle=xy+yz+xz$, so that
$\min\{3x^2~|~x\in\{-1,1\}\}=3$ while $\min\{xy+yz+xz~|~x,y,z\in\{-1,1\}\}=-1$: the two optimal values differ.}
\end{example}

From Theorem~\ref{thm1}, we obtain the following corollary.

\begin{corollary}\label{corol1}
Let $\alpha\geq 3\sqrt{2}\|\T\|_F$.
If $\x^*$ is an optimal solution of \eqref{e4 CSPOP alpha} (equivalently \eqref{e5 H CSPOP alpha}), then $(\x^*,\x^*,\x^*)$ is an optimal solution of \eqref{e6 H CSMOP alpha}.
Consequently, \eqref{e4 CSPOP alpha} (or equivalently \eqref{e5 H CSPOP alpha}) and \eqref{e6 H CSMOP alpha} share at least one common optimal solution.
\end{corollary}

\proof For problem \eqref{e4 CSPOP alpha}, it has at least one optimal solution in $D$ since the continuous objective function must have an optimal solution on a compact set.
Suppose that $\x^*$ is an optimal solution of \eqref{e4 CSPOP alpha}. Then, $\x=\y=\z=\x^*$ is an optimal solution of \eqref{e6 H CSMOP alpha} and the desired result holds.
\qed

Corollary~\ref{corol1} shows that an optimal solution of \eqref{e4 CSPOP alpha} induces an optimal solution of \eqref{e6 H CSMOP alpha}.
However, the converse question arises naturally.
\begin{quote}
If \eqref{e6 H CSMOP alpha} admits an optimal solution $(\x^*,\y^*,\z^*)$ with $\x^*\neq\y^*$ or $\y^*\neq\z^*$, how can we recover an optimal solution of the original problems \eqref{e1 cspop}, \eqref{e4 CSPOP alpha}, or \eqref{e5 H CSPOP alpha}?
\end{quote}
To the best of our knowledge, this question has not been addressed in the existing literature.
The following theorem provides a solution.

\begin{theorem}\label{them2}
Let $\alpha\geq 3\sqrt{2}\|\T\|_F$.
Suppose $(\x^*,\y^*,\z^*)\in D\times D\times D$ is an optimal solution of \eqref{e6 H CSMOP alpha}.
Then $\x^*,\y^*,\z^*$ are all optimal solutions of \eqref{e1 cspop} (equivalently \eqref{e4 CSPOP alpha} and \eqref{e5 H CSPOP alpha}), i.e.,
\[
\langle\T,\tilde{\x}^*\circ\tilde{\x}^*\circ\tilde{\x}^*\rangle
=
\langle\T,\tilde{\y}^*\circ\tilde{\y}^*\circ\tilde{\y}^*\rangle
=
\langle\T,\tilde{\z}^*\circ\tilde{\z}^*\circ\tilde{\z}^*\rangle=\T_3^*.
\]
Here $\tilde{\x}^*=(1,{\x^*}^\top)^\top$, $\tilde{\y}^*=(1,{\y^*}^\top)^\top$, and $\tilde{\z}^*=(1,{\z^*}^\top)^\top$.
\end{theorem}

\proof To prove the conclusion, it is enough to prove that $\x^*, \y^*, \z^*\in D$ are optimal solutions of \eqref{e4 CSPOP alpha}. First of all,
by Lemma \ref{lema1} and the condition
$$
(\x^*, \y^*, \z^*)=\arg\min\{\tilde{F}(\x,\y,\z)~|~\x, \y, \z\in D\},
$$
we know that
$$
\begin{aligned}
0&\geq\langle\T-\alpha\bar{\T}, \tilde{\x}^*\circ(\tilde{\y}^*-\tilde{\z}^*)\circ(\tilde{\y}^*-\tilde{\z}^*)\rangle\\
&=\langle\T-\alpha\bar{\T}, \tilde{\x}^*\circ\tilde{\y}^*\circ\tilde{\y}^*\rangle+\langle\T-\alpha\bar{\T}, \tilde{\x}^*\circ\tilde{\z}^*\circ\tilde{\z}^*\rangle-2\langle\T-\alpha\bar{\T}, \tilde{\x}^*\circ\tilde{\y}^*\circ\tilde{\z}^*\rangle,
\end{aligned}
$$
which means that
$$
2\langle\T-\alpha\bar{\T}, \tilde{\x}^*\circ\tilde{\y}^*\circ\tilde{\z}^*\rangle\geq\langle\T-\alpha\bar{\T}, \tilde{\x}^*\circ\tilde{\y}^*\circ\tilde{\y}^*\rangle+\langle\T-\alpha\bar{\T}, \tilde{\x}^*\circ\tilde{\z}^*\circ\tilde{\z}^*\rangle.
$$
Combining this with the fact that
$$
\begin{aligned}
&\langle\T-\alpha\bar{\T}, \tilde{\x}^*\circ\tilde{\y}^*\circ\tilde{\y}^*\rangle\geq\langle\T-\alpha\bar{\T}, \tilde{\x}^*\circ\tilde{\y}^*\circ\tilde{\z}^*\rangle=\min\{\tilde{F}(\x,\y,\z)~|~\x, \y, \z\in D\},\\
&\langle\T-\alpha\bar{\T}, \tilde{\x}^*\circ\tilde{\z}^*\circ\tilde{\z}^*\rangle\geq\langle\T-\alpha\bar{\T}, \tilde{\x}^*\circ\tilde{\y}^*\circ\tilde{\z}^*\rangle=\min\{\tilde{F}(\x,\y,\z)~|~\x, \y, \z\in D\},
\end{aligned}
$$
we have that
{
\begin{equation}\label{e10}
\langle\T-\alpha\bar{\T}, \tilde{\x}^*\circ\tilde{\y}^*\circ\tilde{\y}^*\rangle=\langle\T-\alpha\bar{\T}, \tilde{\x}^*\circ\tilde{\z}^*\circ\tilde{\z}^*\rangle=\langle\T-\alpha\bar{\T}, \tilde{\x}^*\circ\tilde{\y}^*\circ\tilde{\z}^*\rangle=\T_3^*.
\end{equation}
}
This implies that $(\x^*, \y^*)$ and $(\x^*, \z^*)$ are optimal solutions of {$\min\{\tilde{F}(\x,\y,\y)~|~\x, \y\in D\}$}.

{On the other hand, applying Lemma \ref{lema1} along $\tilde{\x}^*-\tilde{\y}^*$ exactly as in the first part of the proof of \Cref{thm1} (with $(\x,\y)$ replaced by $(\x^*,\y^*)$) gives}
$
3\langle\T-\alpha\bar{\T}, \tilde{\x}^*\circ\tilde{\y}^*\circ\tilde{\y}^*\rangle\geq\langle\T-\alpha\bar{\T}, \tilde{\x}^*\circ\tilde{\x}^*\circ\tilde{\x}^*\rangle+2\langle\T-\alpha\bar{\T}, \tilde{\y}^*\circ\tilde{\y}^*\circ\tilde{\y}^*\rangle.
$
{Since $(\x^*,\y^*)$ minimizes $\tilde{F}(\x,\y,\y)$ over $D\times D$, we also have} $\langle\T-\alpha\bar{\T}, \tilde{\x}^*\circ\tilde{\x}^*\circ\tilde{\x}^*\rangle\geq\langle\T-\alpha\bar{\T}, \tilde{\x}^*\circ\tilde{\y}^*\circ\tilde{\y}^*\rangle$ and $\langle\T-\alpha\bar{\T}, \tilde{\y}^*\circ\tilde{\y}^*\circ\tilde{\y}^*\rangle\geq\langle\T-\alpha\bar{\T}, \tilde{\x}^*\circ\tilde{\y}^*\circ\tilde{\y}^*\rangle$. Combining the three inequalities yields
\begin{equation}\label{e11}
\langle\T-\alpha\bar{\T}, \tilde{\x}^*\circ\tilde{\y}^*\circ\tilde{\y}^*\rangle=\langle\T-\alpha\bar{\T}, \tilde{\x}^*\circ\tilde{\x}^*\circ\tilde{\x}^*\rangle=\langle\T-\alpha\bar{\T}, \tilde{\y}^*\circ\tilde{\y}^*\circ\tilde{\y}^*\rangle.
\end{equation}
{By the same argument applied to the pair $(\x^*,\z^*)$,}
\begin{equation}\label{e12}
\langle\T-\alpha\bar{\T}, \tilde{\x}^*\circ\tilde{\z}^*\circ\tilde{\z}^*\rangle=\langle\T-\alpha\bar{\T}, \tilde{\x}^*\circ\tilde{\x}^*\circ\tilde{\x}^*\rangle=\langle\T-\alpha\bar{\T}, \tilde{\z}^*\circ\tilde{\z}^*\circ\tilde{\z}^*\rangle.
\end{equation}
Furthermore, from \eqref{e10}-\eqref{e12}, we have
$$
\langle\T-\alpha\bar{\T},\tilde{\x}^*\circ\tilde{\x}^*\circ\tilde{\x}^*\rangle=\langle\T-\alpha\bar{\T},\tilde{\y}^*\circ\tilde{\y}^*\circ\tilde{\y}^*\rangle=\langle\T-\alpha\bar{\T},\tilde{\z}^*\circ\tilde{\z}^*\circ\tilde{\z}^*\rangle,
$$
which implies that $\x^*, \y^*, \z^*$ are all optimal solutions of \eqref{e4 CSPOP alpha} (or \eqref{e1 cspop} and \eqref{e5 H CSPOP alpha}), and the desired result holds.
\qed

\subsection{Numerical Algorithm for CSPOP}
\label{sec pam algo}

In this subsection, we develop a numerical algorithm for solving \eqref{e1 cspop}.
In recent years, the proximal alternating minimization (PAM) algorithm has been extensively studied for spherical-constrained homogeneous polynomial optimization problems due to its strong empirical performance and numerical stability (see, e.g., \cite{chen2022efficient,chen2025tensor}). From Theorems \ref{thm1} and \ref{them2}, under suitable assumptions (i.e., for sufficiently large $\alpha>0$), it follows that the optimal value of \eqref{e1 cspop} can be recovered from the optimal value of the cubic spherical multilinear optimization problem \eqref{e6 H CSMOP alpha}.
Motivated by this equivalence and the multi-block structure of \eqref{e6 H CSMOP alpha}, we extend the PAM framework from homogeneous polynomial optimization to the present inhomogeneous polynomial setting. Specifically, we adopt a block coordinate descent scheme in which each block variable is updated sequentially while keeping the others fixed. The resulting algorithm can be viewed as a PAM-type method tailored to the multilinear structure of the CSMOP. {Its convergence properties are established in \Cref{thm3}.}

We present the algorithm in \Cref{alg1}. {\Cref{thm3} proves monotonic decrease, square-summability of the increments, and stationarity of every accumulation point.}

\begin{algorithm}[!htbp]
	\caption{(Proximal Alternating Minimization Algorithm for \eqref{e6 H CSMOP alpha}).}\label{alg1}
	\begin{algorithmic}[1]
		\STATE Let $\alpha\geq3\sqrt{2}\|\T\|_F$, {tolerance $\varepsilon_{\rm in}>0$}, ${\bf x}^{(0)},\y^{(0)},\z^{(0)}\in D$, {proximal weights $\beta_i> 0$}$,\;(i=1,2,3)$.
		\FOR{$\kappa =0,1,2,\ldots$}
		\STATE Update $({\bf x}^{(\kappa+1)},\y^{(\kappa+1)},\z^{(\kappa+1)})$ sequentially via
		\begin{subnumcases}{\label{PAMA}}
		{\bf x}^{(\kappa+1)}=\arg\min\limits_{\|{\bf x}\|=1}\left[\tilde{F}({\bf x},\y^{(\kappa)},\z^{(\kappa)})+{\frac{\beta_1}{2}}\|{\bf x}-{\bf x}^{(\kappa)}\|^2\right],\label{PAMA-1} \\
		\y^{(\kappa+1)}=\arg\min\limits_{\|\y\|=1}\left[\tilde{F}({\bf x}^{(\kappa+1)},\y,\z^{(\kappa)})+{\frac{\beta_2}{2}}\|\y-\y^{(\kappa)}\|^2\right], \label{PAMA-2}\\
		\z^{(\kappa+1)}=\arg\min\limits_{\|\z\|=1}\left[\tilde{F}({\bf x}^{(\kappa+1)},\y^{(\kappa+1)},\z)+{\frac{\beta_3}{2}}\|\z-\z^{(\kappa)}\|^2\right]. \label{PAMA-3}
		\end{subnumcases}
		\STATE Find $\u^{(\kappa+1)}\in\arg\min\left\{\tilde{T_3}({\bf x}^{(\kappa+1)}),\tilde{T_3}({\bf y}^{(\kappa+1)}),\tilde{T_3}({\bf z}^{(\kappa+1)})\right\}$.
		\STATE {Stop when the maximum projected block KKT residual is at most $\varepsilon_{\rm in}$, and return $\u^{(\kappa+1)}$.}
		\ENDFOR
	\end{algorithmic}
\end{algorithm}

\begin{remark}
Algorithm \ref{alg1} is straightforward to implement under spherical constraints. Since the linear independence constraint qualification automatically holds, the (local) optimal solution for each subproblem of Algorithm \ref{alg1} is guaranteed to be a KKT point. Considering the $\x$-subproblem \eqref{PAMA-1},
\begin{eqnarray}
\label{subprob of PAM}
{\bf x}^{(\kappa+1)}=\arg\min\limits_{\|{\bf x}\|=1}\left[\tilde{F}({\bf x},\y^{(\kappa)},\z^{(\kappa)})+{\frac{\beta_1}{2}}\|{\bf x}-{\bf x}^{(\kappa)}\|^2\right].
\end{eqnarray}
{
After discarding terms independent of $\x$, the objective in
\eqref{subprob of PAM} is linear on the sphere. Define
\[
\mathbf q_x^{(\kappa)}
=(\0,I_n)(\T-\alpha\bar{\T})
 \tilde\y^{(\kappa)}\tilde\z^{(\kappa)}
-\beta_1\x^{(\kappa)}.
\]
If $\mathbf q_x^{(\kappa)}\ne0$, its unique block minimizer is
\begin{equation}
\label{closed form sol}
\x^{(\kappa+1)}=-\frac{\mathbf q_x^{(\kappa)}}
{\|\mathbf q_x^{(\kappa)}\|}.
\end{equation}
Thus the sign is fixed by minimization; it is not an arbitrary choice between
two stationary points. If $\mathbf q_x^{(\kappa)}=0$, the previous block may
be retained. The $\y$- and $\z$-updates have the same form, using the newest
available blocks. Exact minimization of these proximal block models, rather
than the sign of the unregularized multilinear objective, supplies the descent
inequality used below.
}
\end{remark}

Now, we first establish the subsequential convergence of the sequence $\{\t^{(\kappa)}=({\bf x}^{(\kappa)}, \y^{(\kappa)}, \z^{(\kappa)})\}$ generated by Algorithm \ref{alg1}.
For the sake of simplicity, denote $\t=({\bf x},\y,\z)$ and $\tilde{F}({\bf x}, \y, \z)=\tilde{F}(\t)$.
\begin{theorem}\label{thm3}{\rm (Subsequential Convergence)} Let $\{\t^{(\kappa)}\}$ be an infinite sequence generated by Algorithm \ref{alg1}.
Let ${\bar{\beta}:=\min\{\beta_1,\beta_2,\beta_3\}>0}$. Then the following statements hold.

\begin{enumerate}
    \item The function sequence $\{\tilde{F}(\t^{(\kappa)})\}$ is nonincreasing and convergent, and
{$\sum_{\kappa=0}^{+\infty}\|\t^{(\kappa+1)}-\t^{(\kappa)}\|^2<+\infty$. In particular, $\|\t^{(\kappa+1)}-\t^{(\kappa)}\|\to0$.}

    \item  Suppose $\bar{\t}$ is a cluster point of $\{\t^{(\kappa)}\}$. Then $\bar{\t}$ is a KKT point of \eqref{e6 H CSMOP alpha} and
$\lim_{\kappa \rightarrow+\infty}\tilde{F}(\t^{(\kappa)})=\tilde{F}(\bar{\t})$. {Since the linear independence constraint qualification holds everywhere on $D\times D\times D$, $\bar{\t}$ is equivalently a stationary point of \eqref{e6 H CSMOP alpha}.}
\end{enumerate}
\end{theorem}
\proof
By Algorithm \ref{alg1}, it follows that
$$
\begin{aligned}
&\tilde{F}(\t^{(\kappa+1)})+\frac{\beta_3}{2}\|\z^{(\kappa+1)}-\z^{(\kappa)}\|^2\leq \tilde{F}({\bf x}^{(\kappa+1)},\y^{(\kappa+1)},\z^{(\kappa)}),\\
&\tilde{F}({\bf x}^{(\kappa+1)},\y^{(\kappa+1)},\z^{(\kappa)})+\frac{\beta_2}{2}\|\y^{(\kappa+1)}-\y^{(\kappa)}\|^2\leq \tilde{F}({\bf x}^{(\kappa+1)},\y^{(\kappa)},\z^{(\kappa)}),\\
&\tilde{F}({\bf x}^{(\kappa+1)},\y^{(\kappa)},\z^{(\kappa)})+\frac{\beta_1}{2}\|{\bf x}^{(\kappa+1)}-{\bf x}^{(\kappa)}\|^2\leq \tilde{F}(\t^{(\kappa)}),
\end{aligned}
$$
which implies
\begin{equation}\label{e16}
\tilde{F}(\t^{(\kappa+1)})+\frac{{\bar{\beta}}}{2}\|\t^{(\kappa+1)}-\t^{(\kappa)}\|^2\leq \tilde{F}(\t^{(\kappa)}).
\end{equation}
Therefore, the sequence $\{\tilde{F}(\t^{(\kappa)})\}$ is nonincreasing. Since $\tilde{F}(\t)$ is bounded below on the compact set $D\times D\times D$, it converges. {Summing \eqref{e16} proves square-summability of the increments, which is the conclusion directly justified by this descent inequality.} The conclusion 1 follows.

The sequence $\{\t^{(\kappa)}\}$ is bounded and therefore has cluster points. Suppose that $\{\t^{(\kappa_j)}\}$ is a subsequence of  $\{\t^{(\kappa)}\}$ converging to $\bar{\t}=(\bar{{\bf x}}, \bar{\y},\bar{\z})$. By the argument above, $\lim_{\kappa \rightarrow+\infty}\tilde{F}(\t^{(\kappa)})=\tilde{F}(\bar{\t})$. On the other hand, by \eqref{e16}, it holds that
$$
\lim_{\kappa \rightarrow+\infty}\|\t^{(\kappa+1)}-\t^{(\kappa)}\|=0.
$$
Then, we have that $\lim_{j\rightarrow+\infty}\t^{(\kappa_j+1)}=\bar{\t}$.
{
For each $j$, the block optimality conditions admit multipliers
$\lambda_{1,j},\lambda_{2,j},\lambda_{3,j}$. For example,
\[
\nabla_{\x}\tilde F(\x^{(\kappa_j+1)},\y^{(\kappa_j)},\z^{(\kappa_j)})
+\beta_1(\x^{(\kappa_j+1)}-\x^{(\kappa_j)})
=\lambda_{1,j}\x^{(\kappa_j+1)},
\]
with analogous equations for $\y$ and $\z$. Taking the inner product with
the corresponding unit block shows that the multiplier sequences are
bounded. Passing to convergent multiplier subsequences, using continuity and
the vanishing increments, yields the three sphere-constrained KKT equations
at $\bar\t$. This proves the second statement.
}
\qed

\section{Convergence and Complexity Analysis}
\label{sec: convergence}

In this section we establish the global convergence and evaluation complexity of \Cref{HTM algo}. The algorithm reduces $\|\nabla_{\x }f(\x_k)\|$ below $\epsilon$ by operating in \emph{two phases}, separated by the gradient threshold $\epsilon^a$. Throughout this section we fix
\[
a:=\frac{p+1}{2p},\qquad b:=\frac{a}{2}=\frac{p+1}{4p},
\]
so that for $p=3$ one has $a=\tfrac23$ and $b=\tfrac13$.

\begin{itemize}[topsep=2pt,itemsep=2pt,parsep=0pt,leftmargin=*]
\item \textbf{Far-field phase} ($\|\mathbf{g}_k\|\ge\epsilon^a$, \Cref{sec far field}). A homogeneous-tensor (HTM) step on the trust-region boundary, accepted through the ratio test, guarantees a per-step decrease of order $\epsilon^{(p+1)/p}$. Summed over the phase this yields at most $O(\epsilon^{-(p+1)/p})$ iterations.
\item \textbf{Near-stationary phase} ($\epsilon\le\|\mathbf{g}_k\|\le\epsilon^a$). The radius is selected according to the smallest Hessian eigenvalue $\lambda_{\min}(\mathbf{H}_k)$, giving three regimes:
\begin{itemize}
\item \emph{negative curvature} ($\lambda_{\min}(\mathbf{H}_k)\le-\epsilon^b$, \Cref{sec:NSNC}): a boundary step extracts decrease from the quadratic term;
\item \emph{weak curvature} ($-\epsilon^b<\lambda_{\min}(\mathbf{H}_k)\le\epsilon^b$, \Cref{sec:NSWC}): decrease is driven by the cubic term, under the nondegeneracy condition of \Cref{ass:weak-curvature};
\item \emph{convex} ($\lambda_{\min}(\mathbf{H}_k)\ge\epsilon^b$, \Cref{sec Refinement}): the Taylor model is locally convex and {a quadratically convergent Newton-type refinement} drives $\|\nabla_{\x} f\|$ below $\epsilon$.
\end{itemize}
\end{itemize}

{The far-field phase dominates the iteration count, so the overall evaluation complexity is the AR$p$-optimal $O(\epsilon^{-(p+1)/p})$, while the near-stationary regimes contribute strictly lower-order terms. This is the precise sense in which \Cref{HTM algo} surpasses the $O(\epsilon^{-2})$ barrier of a naive trust-region method on a $p\ge3$ Taylor model~\cite[Sec.~4.1.3]{cartis2022evaluation}: there, a radius rule of the form $\Delta_k\propto\|\mathbf{g}_k\|$ forces quadratic-type $\Theta(\|\mathbf{g}_k\|^2)$-decrease steps all the way down to $\|\mathbf{g}_k\|=\epsilon$, which sums to $\epsilon^{-2}$; here, those expensive steps are taken only until $\|\mathbf{g}_k\|=\epsilon^a$, summing to $\epsilon^{-2a}=\epsilon^{-(p+1)/p}$, and the residual stretch $\|\mathbf{g}_k\|\in[\epsilon,\epsilon^a]$ is of strictly lower order. In this sense Ada-HTM is a trust-region method only in its globalization: its per-step model decrease and radius schedule match those of the optimal AR$p$ family, and the resulting rate attains the lower bound $\Omega(\epsilon^{-(p+1)/p})$, valid for any method using derivatives up to order $p$~\cite{carmon2020lower,carmon2021lower}. We emphasize that this is not a redefinition of the tolerance: the cheap near-stationary tail is justified by fast local convergence, exactly as in the local analysis of AR$p$ methods. Finally, we solve the boundary subproblem $\min_{\|\s\|=\Delta_k}T_p(\x_k,\s)$ rather than the interior problem $\min_{\|\s\|\le\Delta_k}T_p$, since the latter is NP-hard for $p\ge3$~\cite{nesterov2003random}, whereas the radius-controlled boundary problem is exactly the cubic spherical problem solved by \Cref{alg1}. \Cref{tab:roadmap} collects the per-regime guarantees together with the precise references.}

\begin{table}[h]
\centering
\caption{Roadmap of \Cref{sec: convergence}. Constants $a=\frac{p+1}{2p}$, $b=\frac{p+1}{4p}$; for $p=3$, $a=\frac23$, $b=\frac13$. ``Dominated'' means the iteration count is of strictly lower order than the far-field bound, so it does not affect the overall complexity.}
\label{tab:roadmap}
\renewcommand{\arraystretch}{1.3}
\begin{tabular}{p{3cm}p{3.1cm}p{3cm}p{2.2cm}p{2.2cm}}
\hline
Regime & Condition & Radius $\Delta_k$ & Decrease/step & \#\,iterations\\
\hline
Far-field & $\|\mathbf{g}_k\|\ge\epsilon^a$ & ${\ge\kappa_\Delta\epsilon^a}$ & $\epsilon^{(p+1)/p}$ & $O(\epsilon^{-\frac{p+1}{p}})$ \newline (\Cref{thm far complexity})\\
Negative curvature & $\lambda_{\min}(\mathbf{H}_k)\le-\epsilon^b$ & ${C_{\rm ns}\epsilon^b}$ & $\epsilon^{\frac34\frac{p+1}{p}}$ & dominated \newline (\Cref{thm:NSNC})\\
Weak curvature & $|\lambda_{\min}(\mathbf{H}_k)|<\epsilon^b$ & ${C_{\rm ns}\epsilon^b}$ & $\epsilon^{\frac34\frac{p+1}{p}}$ & dominated \newline (\Cref{thm:NSWC})\\
Convex & $\lambda_{\min}(\mathbf{H}_k)\ge\epsilon^b$ & local convex region $\Omega_k$ & $\|\nabla_{\x} f\|\!\to\!\epsilon$ & ${O(\log\log(1/\epsilon))}$ \newline ({\Cref{cor:newton-refine}})\\
\hline
\end{tabular}
\end{table}

\begin{remark}[Only needs sufficient model decrease and an approximate subproblem solver]
\label{rem:inexact-pam}
{
The complexity guarantee of \Cref{HTM algo} does not require a \emph{global}
minimizer of $\min_{\|\s\|=\Delta_k}T_p(\x_k,\s)$. PAM generates a
stationary candidate, but stationarity or a small inner KKT residual alone
does not imply the model comparison required by the outer proofs. The
reference safeguard \eqref{eq:safeguarded-boundary-step} supplies that
comparison deterministically: in the far field it uses a gradient step, and
in the nonconvex near field it uses a signed curvature direction. Thus the
stationary-point result for PAM and the evaluation-complexity result for
Ada--HTM are logically separate but algorithmically compatible.
}
\end{remark}

\subsection{Assumptions and technical preliminaries}
\label{sec:assumptions}

\begin{assumption}
$f\in \C^{p, 1}(\R^n)$ with $p \ge 3$ which means that $f$ is $p$-times continuously differentiable, bounded below by a constant $f_\text{low}$ and the $p$th derivative of $f$ is globally Lipschitz continuous. Namely, there exists a constant $L > 0$ such that, for all
$\x, \y \in \R^n$, $\|\nabla_{\x}^p f(\x) - \nabla_{\x}^p f(\y)\| \le (p-1)! L \|\x-\y\|.$
\label{assumption Liptz}
\end{assumption}

\begin{definition}[{Tensor operator norm}]
\label{def:tensor-norm}
{
For a $j$th-order tensor $T \in \mathbb{R}^{n\times\cdots\times n}$, define
\[
\|T\|_{\rm op}:=
\max_{\|v_1\|=\cdots=\|v_j\|=1}|T[v_1,\ldots,v_j]|.
\]
This norm directly controls the multilinear contractions used in the Taylor
bounds below; an entrywise maximum would introduce dimension-dependent
factors and is therefore insufficient for the stated inequalities.
}
\end{definition}

\begin{assumption}[Uniformly bounded derivatives]
\label{assumption:bounded-derivatives}
Let $f\in C^{p,1}(\mathbb{R}^n)$.
{Let $\mathcal L_0$ be a bounded open set containing the sublevel set
$\{\x : f(\x)\le f(\x_0)\}$ together with all trial points and trial segments
$[\x_k,\x_k+\s_k]$ generated by the algorithm. Assume that the derivatives of $f$ up to
order $p$ are uniformly bounded on $\mathcal L_0$}, namely
\[
{\|\nabla_{\x}^j f(\x)\|_{\rm op} \le \Lambda_j,
\qquad
\forall\, \x \in \mathcal L_0,
\quad j=1,\dots,p,}
\]
where $\Lambda_j>0$ are constants independent of $\x$. {Accepted iterates
remain in the sublevel set by monotonicity of the accepted objective values, and the
radii employed by the algorithm keep the trial segments in $\mathcal L_0$ for all
sufficiently small $\epsilon$. This formulation accommodates polynomial test functions,
whose derivatives are unbounded on all of $\mathbb{R}^n$.}
\end{assumption}

\begin{lemma}(Adapted from \cite[Lemma 2.1]{cartis2020concise})
   Let $f\in \C^{p, 1}(\R^n)$. $T_p(\x, \s)$ is the Taylor approximation of $f (\x + \s)$ about $\x$. Then, under Assumption \ref{assumption Liptz}, for all $\x, \s \in \R^n$,
\begin{eqnarray}
      |f(\x+\s) - T_p(\x, \s) | &\le&  \frac{L}{p+1}\|\s\|^{p+1},
            \label{lip bound 1}
      \\ \|\nabla_{\x} f(\x+\s) - \nabla_{\s} T_p(\x, \s) \|  &\le& L \|\s\|^p.
      \label{lip bound 2}
\end{eqnarray}
\label{tech lemma 1}
\end{lemma}
\begin{proof}
  The proof can be found in \cite[Appendix A.1]{cartis2020concise}.
\end{proof}

We collect here the remaining objects used throughout \Cref{sec: convergence}: a \emph{reference direction}, a genericity condition on the cubic term, and a smallness condition on $\epsilon$.

\begin{definition}[Reference direction]
\label{def:reference-direction}
For an iterate $\x_k$ with Hessian $\mathbf{H}_k$ and radius $\Delta_k$, let $\s_k^{(h)}\in\R^n$ satisfy $\|\s_k^{(h)}\|=\Delta_k$ and $\mathbf{H}_k[\s_k^{(h)}]^2=\lambda_{\min}(\mathbf{H}_k)\Delta_k^2$. Among $\pm\s_k^{(h)}$ we choose the sign so that ${\nabla^3 f(\x_k)[\s_k^{(h)}]^3\le0}$.
\end{definition}

\begin{assumption}[Nondegeneracy of the cubic term]
\label{ass:weak-curvature}
In the weak-curvature regime $-\epsilon^b<\lambda_{\min}(\mathbf{H}_k)\le\epsilon^b$, the cubic term does not vanish along the reference direction of \Cref{def:reference-direction}; that is, $\nabla^3 f(\x_k)[\s_k^{(h)}]^3=-\tilde{\Lambda}_3\Delta_k^3<0$ for some constant $\tilde{\Lambda}_3>0$.
\end{assumption}
\noindent Assumption \ref{ass:weak-curvature} is the only genericity condition in our analysis, and we state it explicitly. It is used solely in the weak-curvature regime: there both the gradient and the Hessian are small, so descent must be supplied by the cubic term, and the assumption guarantees this term is active along the weakest-curvature direction. It is the trust-region counterpart of the nondegeneracy conditions standard in the local analysis of AR$p$ methods.

{
Fix
\begin{equation}
C_{\rm ns}\ge
\max\left\{2p,
\left(\frac{6(p+1)}{\tilde\Lambda_3}\right)^{1/2},
\frac{3(p+1)}{\tilde\Lambda_3}\right\},
\qquad \Delta_k=C_{\rm ns}\epsilon^b
\label{eq:near-radius}
\end{equation}
in both near-stationary nonconvex regimes. Prescribing the radius at this
scale is important: the leading quadratic or cubic term requires a lower
bound on $\Delta_k$, whereas control of the higher-order remainder requires
an upper bound of the same order. For notational uniformity, set
$\Lambda_{p+1}:=Lp!/(1-\eta)$ for the Taylor-remainder contribution, where $L$ is the Lipschitz constant of Assumption~\ref{assumption Liptz} and $\eta\in(0,1)$ is the ratio-test parameter of \Cref{HTM algo}.
}

\begin{assumption}[Smallness of $\epsilon$]
\label{assumption epsilon small}
With $a=\frac{p+1}{2p}$ and $b=\frac{a}{2}$, assume $\epsilon>0$ is small enough that, for every $j=4,\dots,p+1$,
\[
{\frac{\Lambda_j}{j!}C_{\rm ns}^{j-2}\epsilon^{b(j-2)}\le\frac{\epsilon^{b}}{2p}}
\qquad\text{and}\qquad
{\frac{\Lambda_j}{j!}C_{\rm ns}^{j-3}\epsilon^{b(j-3)}\le\frac{\tilde{\Lambda}_3}{6(p+1)}}.
\]
{Both inequalities hold for all sufficiently small $\epsilon$ because $b>0$.} They ensure that the quartic-and-higher Taylor terms are dominated by the quadratic and cubic terms along the reference direction; these are the standard smallness requirements of AR$p$-type local analysis.
\end{assumption}

\subsection{Far Field Regime}
\label{sec far field}

\begin{theorem}[Lower bound on $\Delta_k$ in the far-field regime]
\label{thm:Delta-lower}
Let $f\in C^{p,1}(\mathbb{R}^n)$ and suppose that
$
\|\mathbf g_k\| \ge \epsilon^a ,  0\le a\le 1 .
$
Under Assumptions \ref{assumption Liptz}--\ref{assumption:bounded-derivatives} {and the safeguarded-step rule \eqref{eq:safeguarded-boundary-step}}, if
\[
\Delta_k \le C \min\{\|\mathbf g_k\|,1\}, \qquad \text{where} \qquad C :=
\min_{2\le j\le p+1}
\left(
\frac{(p+1)\Lambda_j}{j!}
\right)^{-1/(j-1)},
\]
{with $\Lambda_{p+1}$ as defined in \Cref{sec:assumptions}.}
Then iteration $k$ is successful and satisfies both        \begin{equation}
        f(\x_k) - T_p(\x_k,\s_k) \ge \frac{\Delta_k \epsilon^a}{p+1} > 0
        \label{Model Decrease test 1}
        \end{equation} and the ratio test~\eqref{ratio test}.
Consequently, the trust-region radius admits the lower bound
\begin{equation}
\label{Delta bound}
{
\Delta_k\ge\Delta_{\min}:=
\min\{\Delta_0,\gamma_1C\epsilon^a\}.
}
\end{equation}
{If $\Delta_0\ge\kappa_0\epsilon^a$ for a constant
$\kappa_0>0$, then $\Delta_{\min}\ge\kappa_\Delta\epsilon^a$ with
$\kappa_\Delta:=\min\{\kappa_0,\gamma_1C\}>0$.}
\end{theorem}

\begin{remark}[The constant $C$ is independent of $k$]\label{rem:C-uniform}
{The constant $C$ in \Cref{thm:Delta-lower} depends only on the order $p$, the ratio-test parameter $\eta$, the Lipschitz constant $L$, and the uniform derivative bounds $\Lambda_2,\dots,\Lambda_p$ of \Cref{assumption:bounded-derivatives}; no iterate-dependent quantity enters its definition. Consequently, {$\Delta_{\min}\ge\kappa_\Delta\epsilon^a$ in \eqref{Delta bound}} is a genuine uniform lower bound, valid at every far-field iteration, which is exactly what the complexity count in \Cref{lemma successful far} requires.}
\end{remark}

\begin{proof}
Assume that $
\Delta_k \le C \min\{\|\textbf{g}_k\|,1\}.
$
Define
\begin{equation}
E_k(\s):=
f(\x_k)-\left(T_p(\x_k,\s)+\frac{L}{(1-\eta)(p+1)}\|\s\|^{p+1}\right).
\label{Ek}
\end{equation}
{Let $\s_k$ be the safeguarded step in \eqref{eq:safeguarded-boundary-step}} and let $\s_k^{(g)}$ be the gradient direction with norm  $\Delta_k$,
\[
\s_k^{(g)}:=-\frac{\textbf{g}_k}{\|\textbf{g}_k\|}\Delta_k.
\]
Since $\|\s_k\|=\|\s_k^{(g)}\|=\Delta_k$ {and $T_p(\x_k,\s_k)\le T_p(\x_k,\s_k^{(g)})$ by the safeguard}, while the additional term $
\frac{L}{(1-\eta)(p+1)}\|\s\|^{p+1}
$
is constant on the sphere $\{\|\s\|=\Delta_k\}$, we have
$
E_k(\s_k)\ge E_k(\s_k^{(g)}).
$
Therefore,
\begin{align*}
E_k(\s_k) \ge E_k(\s_k^{(g)})
&= -\textbf{g}_k^\top \s_k^{(g)}
   - \sum_{j=2}^{p}\frac{1}{j!}\nabla^j f(\x_k)[\s_k^{(g)}]^j
   - \frac{L}{(1-\eta)(p+1)}\|\s_k^{(g)}\|^{p+1}  \ge \|\textbf{g}_k\|\Delta_k
   - \sum_{j=2}^{p+1}\frac{\Lambda_j}{j!}\Delta_k^j.
\end{align*} Using the assumption $\Delta_k \le C\min\{\|\textbf{g}_k\|,1\}$, we obtain
\[
E_k(\s_k)
\ge
\left[
\|\textbf{g}_k\|
-\sum_{j=2}^{p+1}\frac{\Lambda_j}{j!}
C^{\,j-1}\min\{\|\textbf{g}_k\|,1\}^{j-1}
\right]\Delta_k.
\]
By the definition of $C$ in \eqref{Delta bound}, for every $j=2,\dots,p+1$,
$
\frac{\Lambda_j}{j!}C^{\,j-1}\le \frac{1}{p+1}.
$
Hence
{
\begin{align}
E_k(\s_k)
\ge
\left[
\|\textbf{g}_k\|
-\frac{1}{p+1}
\sum_{j=2}^{p+1}\min\{\|\textbf{g}_k\|,1\}^{j-1}
\right]\Delta_k \ge
\left[
\|\textbf{g}_k\|
-\frac{p}{p+1}\min\{\|\textbf{g}_k\|,1\}
\right]\Delta_k \notag\geq
\frac{1}{p+1}\min\{\|\textbf{g}_k\|,1\}\Delta_k.
\label{temp 1 corrected}
\end{align}
}
Since $\|\textbf{g}_k\|\ge \epsilon^a$ and $\epsilon^a\ll 1$, it follows that
$
\min\{\|\textbf{g}_k\|,1\}\ge \epsilon^a.
$
Therefore,
$
E_k(\s_k)\ge \frac{\Delta_k\epsilon^a}{p+1}.
$
By the definition of $E_k$, we deduce that
\begin{equation}
\label{analysis 1}
f(\x_k)-T_p(\x_k,\s_k)
\ge
\frac{\Delta_k\epsilon^a}{p+1}
+
\frac{L}{(1-\eta)(p+1)}\|\s_k\|^{p+1},
\end{equation}
which implies \eqref{Model Decrease test} is successful.
Next, using \eqref{lip bound 1}, we obtain
\begin{align}
|\rho_k-1|
=
\frac{|f(\x_k+\s_k)-T_p(\x_k,\s_k)|}
     {|f(\x_k)-T_p(\x_k,\s_k)|}
\le
\frac{\frac{L}{p+1}\|\s_k\|^{p+1}}
     {\frac{L}{(1-\eta)(p+1)}\|\s_k\|^{p+1}}
\le 1-\eta.
\label{analysis for ratio test}
\end{align}
Thus $\rho_k\ge \eta$, and \eqref{ratio test} also holds. Hence, iteration $k$ is successful.
According to the update rule in Algorithm~\ref{HTM algo}, a successful iteration increases the trust-region radius, while an unsuccessful iteration decreases it by a factor $\gamma_1$.
Since any iteration with
\(
\Delta_k \le C \min\{\|\mathbf g_k\|,1\}
\)
is guaranteed to be successful, the radius cannot decrease indefinitely.
{If the radius crosses this threshold after an unsuccessful
iteration, it remains at least $\gamma_1C\epsilon^a$; otherwise it remains at
least its initial value. This proves \eqref{Delta bound}.}
\end{proof}

\begin{remark}
In Algorithm~\ref{HTM algo}, when $\|\mathbf g_k\|\ge \epsilon^a$, practical initial choices for the trust-region radius include
\[
\Delta_0=\|\nabla_{\x} f(\x_0)\|
\qquad\text{or}\qquad
\Delta_0=\min\{\|\nabla_{\x} f(\x_0)\|,1\}.
\]

Moreover, if $\gamma_1 C\le 1$, then these choices satisfy the lower bound in \eqref{Delta bound}.
\end{remark}

\begin{corollary}[Model and function decrease in successful far-field iterations]
\label{cor:model-decrease-far}
Suppose that $\|\mathbf g_k\|\ge \epsilon^a$.
Then, by  \eqref{Model Decrease test}, any successful iteration satisfies
$
f(\x_k)-T_p(\x_k,\s_k)
\ge
\frac{\Delta_k\epsilon^a}{p+1}.
$
Moreover,
\[
f(\x_k)-f(\x_{k+1})
\ge
\kappa_f\,\epsilon^{2a},
\qquad
{\kappa_f:=\frac{\eta\kappa_\Delta}{p+1}}.
\]
In particular, if $a=\frac{p+1}{2p}$, then
$$
f(\x_k)-f(\x_{k+1})
\ge
\kappa_f\,\epsilon^{\frac{p+1}{p}}.
$$
\end{corollary}
\begin{proof}
The first inequality follows directly from \eqref{Model Decrease test}.
For a successful iteration, combining the ratio test with the model decrease
condition, we obtain
\[
f(\x_k)-f(\x_{k+1})
\ge
\eta\big[f(\x_k)-T_p(\x_k,\s_k)\big]
\ge
\frac{\eta\Delta_k\epsilon^a}{p+1} \ge
{\frac{\eta\kappa_\Delta\epsilon^{2a}}{p+1}}
\]
where the last inequality uses ${\Delta_k\ge\kappa_\Delta\epsilon^a}$ (\Cref{thm:Delta-lower}).
Substituting $a=\frac{p+1}{2p}$ gives the required bound.
\end{proof}

\begin{lemma}
\label{lemma successful far}
\textbf{(Bound on successful iterations in the far-field regime)}
{Let $\mathcal S$ and $\mathcal U$ denote the index sets of successful and unsuccessful far-field iterations, respectively}, and let $f_{\mathrm{low}}$ be a lower bound of $f$. Then
\[
|\mathcal S|
\le
\frac{f(\x_0)-f_{\mathrm{low}}}{\kappa_f}
\epsilon^{-\frac{p+1}{p}}.
\]
\end{lemma}
\begin{proof}
{Each successful far-field iteration decreases $f$ by at least $\kappa_f\epsilon^{\frac{p+1}{p}}$ by \Cref{cor:model-decrease-far}, the objective is nonincreasing along the iterations, and $f\ge f_{\mathrm{low}}$; summing the decreases gives the bound.}
\end{proof}

\begin{lemma}
\label{lemma unsuccessful}
\textbf{(Bound on unsuccessful iterations in the far-field regime)}
It holds that
{
\[
|\mathcal U|
\le
\frac{\log \gamma_2}{|\log \gamma_1|}\,|\mathcal S|
+
\frac{1}{|\log \gamma_1|}\log\!\Big(\frac{\Delta_0}{\Delta_{\min}}\Big).
\]
}
\end{lemma}
\begin{proof}
{Every successful iteration multiplies the radius by $\gamma_2$ and every unsuccessful one by $\gamma_1<1$, while \Cref{thm:Delta-lower} guarantees $\Delta_k\ge\Delta_{\min}$ throughout the far-field phase. Hence $\Delta_{\min}\le\Delta_0\,\gamma_2^{|\mathcal S|}\gamma_1^{|\mathcal U|}$, and taking logarithms gives the stated bound.}
\end{proof}

\begin{theorem}
\label{thm far complexity}
\textbf{(Far-field complexity bound)}
{Under Assumptions~\ref{assumption Liptz}--\ref{assumption:bounded-derivatives}, the safeguarded-step condition \eqref{eq:safeguarded-boundary-step}, and $\Delta_0\ge\kappa_0\epsilon^a$,} Algorithm~\ref{HTM algo} requires at most
\[
\mathcal O(\epsilon^{-\frac{p+1}{p}})
\]
iterations in the far-field regime before entering the near-stationary regime.
\end{theorem}
\begin{proof}
{The total number of far-field iterations is $|\mathcal S|+|\mathcal U|$, which is $O(\epsilon^{-\frac{p+1}{p}})$ by \Cref{lemma successful far,lemma unsuccessful}, since $\Delta_{\min}\ge\kappa_\Delta\epsilon^a$ contributes only an $O(\log(1/\epsilon))$ additive term.}
\end{proof}

\subsection{Near-Stationary Negative Curvature Regime (NSNC)}
\label{sec near field}

We now consider the NSNC regime where the gradient is small but first-order stationarity has not yet been reached, such that
\[
\epsilon \le \|\mathbf g_k\| \le \epsilon^a,
\qquad  \lambda_{\min}(\mathbf H_k)\le \epsilon^b
\]
and $a:=\frac{p+1}{2p}$, and $b:=\frac{a}{2}=\frac{p+1}{4p}$.
In this regime, the behavior of the Taylor model depends on the curvature of the Hessian.
When the Hessian exhibits sufficiently negative curvature, a suitable boundary step yields a sufficient decrease through the quadratic term.
When the Hessian is nearly singular, the decrease is instead driven by the cubic term, under {an explicit} nondegeneracy condition.
Accordingly, we distinguish two subcases.
Section~\ref{sec:NSNC} studies the \emph{negative-curvature regime}, where
\[
\lambda_{\min}(\mathbf H_k)\le -\epsilon^b,
\]
while Section~\ref{sec:NSWC} studies the \emph{weak-curvature regime}, where
\[
-\epsilon^b < \lambda_{\min}(\mathbf H_k)\le \epsilon^b.
\]
In both cases, we show that an appropriate choice of the trust-region radius ensures that the ratio test is successful and yields a quantifiable reduction in the objective function. {In both regimes the radius is prescribed as $\Delta_k=C_{\rm ns}\epsilon^b=\Theta(\epsilon^{a/2})$, with $C_{\rm ns}$ as in \eqref{eq:near-radius}.}

\subsubsection{Negative-curvature regime}
\label{sec:NSNC}

The proofs of Theorems~\ref{thm:NSNC} and \ref{thm:NSWC} follow the same general pattern.
We first establish a lower bound on the model decrease along the reference direction in Definition~\ref{def:reference-direction}, then show that this decrease dominates the Taylor remainder, which guarantees success of the ratio test, and finally deduce the corresponding reduction in the objective function. {Both proofs use the merit function $E_k$ defined in \eqref{Ek} and conclude via the ratio-test bound \eqref{analysis for ratio test}.}

\begin{theorem}\label{thm:NSNC}
\textbf{(Function value reduction in the negative-curvature regime)}
Let $f\in \C^{p, 1}(\R^n)$, $a: =\frac{p+1}{2p}$ and $b = \frac{a}{2}$. If $\epsilon \le \|\textbf{g}_k\| \le \epsilon^a$ and $\lambda_{\min}(\textbf{H}_k) \le -\epsilon^b$, under Assumptions~\ref{assumption Liptz}--\ref{assumption epsilon small}, for
\begin{eqnarray}
{\Delta_k=C_{\rm ns}\epsilon^b},
  \label{eq:Delta-NSNC}
\end{eqnarray}
 the mechanism of  \Cref{HTM algo} guarantees that the $k$th iteration is successful and
\begin{eqnarray}
 f(\x_k) -f(\x_{k+1}) \ge {\frac{\eta C_{\rm ns}^2}{2p}}\,\epsilon^{\frac{3}{4}\frac{p+1}{p}},\qquad\text{i.e. of order }\epsilon^{\frac{3}{4}\frac{p+1}{p}}\succ\epsilon^{\frac{p+1}{p}},
\end{eqnarray}
for sufficiently small $\epsilon>0$.
\end{theorem}

\begin{proof}
Let $E_k$ be the merit function defined in \eqref{Ek}, let ${\s_k}$ be the safeguarded step in \eqref{eq:safeguarded-boundary-step}, and let $\s_k^{(h)}$ be the reference direction of \Cref{def:reference-direction}. {Because the two steps have the same norm, the safeguard gives $E_k (\s_k) \ge E_k (\s_k^{(h)})$.}
We have
\begin{eqnarray}
    E_k (\s_k)   &\ge&  E_k (\s_k^{(h)}) = -\textbf{g}_k^\top
    \s_k^{(h)} - \frac{1}{2} \overbrace{\textbf{H}_k[\s_k^{(h)}]^2}^{\le- \epsilon^b \Delta_k^2} - \frac{1}{6}  \overbrace{\nabla^3f(\x_k) [\s_k^{(h)}]^3}^{\le 0} - \sum_{4 \le j \le p} \frac{\nabla^j f(\x_k)}{j!} [\s_k^{(h)}]^j - \frac{L\|\s_k^{(h)}\|^{p+1}}{(1-\eta) (p+1)}  \notag
    \\ &\ge& -\|\textbf{g}_k\|\Delta_k + \frac{\epsilon^b}{2} \Delta_k^2  - \sum_{4 \le j \le p+1} \frac{\Lambda_j}{j!} \Delta_k^j  \notag
\\ &\ge& \bigg( \frac{\epsilon^b}{2p} \Delta_k-\|\textbf{g}_k\|\bigg) \Delta_k +  \frac{\epsilon^b}{2p} \Delta_k^2 + \sum_{4 \le j \le p+1} \bigg( \frac{\epsilon^b}{2p}  -\frac{\Lambda_j}{j!} \Delta_k^{j-2}\bigg) \Delta_k^2.
\label{bracket 1}
\end{eqnarray}
{Using \eqref{eq:Delta-NSNC}, $a=2b$, and $C_{\rm ns}\ge2p$, the first bracket in \eqref{bracket 1} is nonnegative.}
For the second bracket, for all $j = 4, \dotsc, p+1$,
$$
 {\frac{\epsilon^b}{2p}  -\frac{\Lambda_j}{j!} \Delta_k^{j-2}
 =\frac{\epsilon^b}{2p}-\frac{\Lambda_j}{j!}C_{\rm ns}^{j-2}\epsilon^{b(j-2)}\ge0},
$$
by \Cref{assumption epsilon small}.
Therefore,
$$
{E_k (\s_k)  \ge \frac{\epsilon^b}{2p} \Delta_k^2
=\frac{C_{\rm ns}^2}{2p}\epsilon^{3b}  >0},
$$
and
$$
{f(\x_k) - T_p(\x_k, \s_k) \ge \frac{C_{\rm ns}^2}{2p}\epsilon^{3b}  +  \frac{L}{(1-\eta) (p+1)} \|\s_k\|^{p+1}}.
$$  Using a similar analysis as \eqref{analysis for ratio test}, we obtain that the iteration $k$ satisfies \eqref{ratio test}. Using \eqref{ratio test}, we have
    \begin{eqnarray*}
 f(\x_k) -f(\x_{k+1}) &\ge& \eta \bigg[ f(\x_k)- T_p(\x_k, \s_k) \bigg]  \ge  {\frac{\eta C_{\rm ns}^2}{2p}\epsilon^{3b}},
\end{eqnarray*}
{and $3b=\tfrac34\tfrac{p+1}{p}$ gives the stated order.}
\end{proof}

\subsubsection{Weak-curvature regime}
\label{sec:NSWC}

\begin{theorem}\label{thm:NSWC}
\textbf{(Function value reduction in the weak-curvature regime)}
Let $f\in \C^{p, 1}(\R^n)$, $a: =\frac{p+1}{2p}$ and $b = \frac{a}{2}$. If $\epsilon \le \|\textbf{g}_k\| \le \epsilon^a$ and $-\epsilon^b \le \lambda_{\min}(\textbf{H}_k) \le \epsilon^b$,
under Assumptions~\ref{assumption Liptz}--\ref{assumption epsilon small}, for
\begin{eqnarray}
  {\Delta_k=C_{\rm ns}\epsilon^b}
  \label{eq:Delta-NSWC},
\end{eqnarray}
 the mechanism of  \Cref{HTM algo} guarantees that the $k$th iteration is successful and
\begin{eqnarray}
 f(\x_k) -f(\x_{k+1}) \ge {c_2 \epsilon^{\frac{3}{4}\frac{p+1}{p}},\qquad
c_2:=\frac{\eta\tilde\Lambda_3C_{\rm ns}^3}{6(p+1)}},
\end{eqnarray}
for sufficiently small $\epsilon>0$.
\end{theorem}

\begin{proof}
Let $E_k$, $\s_k$ and $\s_k^{(h)}$ be as in the proof of \Cref{thm:NSNC}; by \Cref{ass:weak-curvature}, $\nabla^3 f(\x_k)[\s_k^{(h)}]^3=-\tilde{\Lambda}_3\Delta_k^3<0$. {The safeguard and equality of the two norms imply $E_k(\s_k)\ge E_k(\s_k^{(h)})$.} Using $\textbf{H}_k[\s_k^{(h)}]^2=\lambda_{\min}(\textbf{H}_k)\Delta_k^2\le\epsilon^b\Delta_k^2$ and $-\tfrac16\nabla^3 f(\x_k)[\s_k^{(h)}]^3=\tfrac{\tilde{\Lambda}_3}{6}\Delta_k^3$, we obtain
\begin{eqnarray*}
E_k(\s_k) &\ge& E_k(\s_k^{(h)})
= -\textbf{g}_k^\top\s_k^{(h)} - \tfrac12\textbf{H}_k[\s_k^{(h)}]^2 - \tfrac16\nabla^3 f(\x_k)[\s_k^{(h)}]^3 - \sum_{4\le j\le p}\tfrac{\nabla^j f(\x_k)}{j!}[\s_k^{(h)}]^j - \tfrac{L\|\s_k^{(h)}\|^{p+1}}{(1-\eta)(p+1)}\\
&\ge& -\|\textbf{g}_k\|\Delta_k - \tfrac{\epsilon^b}{2}\Delta_k^2 + \tfrac{\tilde{\Lambda}_3}{6}\Delta_k^3 - \sum_{4\le j\le p+1}\tfrac{\Lambda_j}{j!}\Delta_k^j\\
&\ge& \Big(\tfrac{\tilde{\Lambda}_3}{6(p+1)}\Delta_k^2-\|\textbf{g}_k\|\Big)\Delta_k + \Big(\tfrac{\tilde{\Lambda}_3}{6(p+1)}\Delta_k-\tfrac{\epsilon^b}{2}\Big)\Delta_k^2 + \tfrac{\tilde{\Lambda}_3}{6(p+1)}\Delta_k^3 + \sum_{4\le j\le p+1}\Big(\tfrac{\tilde{\Lambda}_3}{6(p+1)}-\tfrac{\Lambda_j}{j!}\Delta_k^{j-3}\Big)\Delta_k^3.
\end{eqnarray*}
{By \eqref{eq:Delta-NSWC}, $a=2b$, and the definition of $C_{\rm ns}$ in \eqref{eq:near-radius}, the first two brackets are nonnegative; \Cref{assumption epsilon small} makes every summand nonnegative.} Hence $E_k(\s_k)\ge\frac{\tilde{\Lambda}_3}{6(p+1)}\Delta_k^3>0$, so
$$ f(\x_k)-T_p(\x_k,\s_k)\ge \frac{\tilde{\Lambda}_3}{6(p+1)}\Delta_k^3 + \frac{L}{(1-\eta)(p+1)}\|\s_k\|^{p+1}. $$
As in \eqref{analysis for ratio test}, iteration $k$ satisfies \eqref{ratio test}; therefore,
$$ f(\x_k)-f(\x_{k+1})\ge \eta\big[f(\x_k)-T_p(\x_k,\s_k)\big]\ge {\frac{\eta\tilde{\Lambda}_3C_{\rm ns}^3}{6(p+1)}\epsilon^{3b}}. $$
Substituting $b=\frac{p+1}{4p}$ gives the desired result.
\end{proof}

\begin{remark}
 Since $b = \frac{a}{2}$, the radius in \eqref{eq:Delta-NSWC} has the order of magnitude
 $
{\Delta_k=\Theta(\epsilon^{\frac{a}{2}})},
 $
 which aligns the negative- and weak-curvature regimes. {The common constant $C_{\rm ns}$ is chosen large enough for the leading-term comparisons and fixed enough to control all higher-order terms.}
\end{remark}

\subsection{Near-stationary convex regime}
\label{sec Refinement}

We finally treat the regime
\[
\epsilon \le \|\mathbf{g}_k\| \le \epsilon^a,\qquad \mathbf{H}_k \succeq c I_n \ \text{with}\ c\ge \epsilon^b ,
\]
in which the iterate is near-stationary and the Hessian is positive definite. The argument has two steps. \textbf{Step 1:} the Taylor model is locally convex on a ball of radius $\Theta(\|\mathbf{g}_k\|)$, so a single local convex minimization is well defined and reduces the gradient to $O(\epsilon)$ (\Cref{lemma local min exists Tp,Thm converge in one step Tp}). \textbf{Step 2:} because the model is locally \emph{strongly} convex, {a quadratically convergent local refinement} then drives the gradient below $\epsilon$ and terminates the algorithm (\Cref{cor:newton-refine}). Define the local region
\begin{equation}
\Omega_k := \left\{\s\in\R^n:\|\s\|\le 4c^{-1}\|\mathbf{g}_k\|\right\}.
\label{bound for Omega Tp}
\end{equation}

\paragraph{Step 1: the model is locally convex and one solve gives an $O(\epsilon)$ gradient.}
\begin{lemma}
\label{lemma local min exists Tp}
Assume that $\textbf{H}_k \succeq cI_n \succ 0$ and
\begin{equation}
\|\textbf{g}_k\|
\le
\frac{c}{4}
\min_{3\le j\le p}
\left\{
\left(
\frac{p\Lambda_j}{(j-2)!c}
\right)^{-\frac{1}{j-2}}
\right\}.
\label{bound for g Tp}
\end{equation}
Then the following statements hold.

{(i)} The Taylor model $T_p(\x_k,\s)$ is convex in $\Omega_k$.

{(ii)} There exists a unique local minimizer $\s_k^*$ of $T_p(\x_k,\s)$ in $\Omega_k$.

{(iii)} This local minimizer satisfies
\[
T_p(\x_k,\s_k^*)\le T_p(\x_k,\mathbf{0})=f(\x_k),
\]
with equality if and only if $\s_k^*=\mathbf{0}$.
\end{lemma}

\begin{proof}
We first prove the convexity of $T_p(\x_k,\s)$ in $\Omega_k$.
For any $\s\in\Omega_k$, the Hessian of the Taylor model is
\[
\nabla^2_{\s}T_p(\x_k,\s)
=
\textbf{H}_k
+
\sum_{j=3}^p \frac{1}{(j-2)!}\nabla_x^j f(\x_k)[\s]^{j-2} \succeq
\left[
c-\sum_{j=3}^p \frac{\Lambda_j}{(j-2)!}\|\s\|^{j-2}
\right]I_n.
\]
{where the last inequality is based on Assumption~\ref{assumption:bounded-derivatives}.}
Since $\|\s\|\le 4c^{-1}\|\textbf{g}_k\|$ and \eqref{bound for g Tp} holds, it follows that
\[
\sum_{j=3}^p \frac{\Lambda_j}{(j-2)!}\|\s\|^{j-2}
\le \frac{p-2}{p}c,\quad \text{and hence} \quad \nabla^2_{\s}T_p(\x_k,\s)\succeq \frac{2c}{p}I_n \succ 0.
\]
Therefore $T_p(\x_k,\s)$ is convex in $\Omega_k$.
Next, we show that the minimizer lies in the interior of $\Omega_k$.
Let $\s_c\in\partial\Omega_k$, so that $\|\s_c\|=4c^{-1}\|\textbf{g}_k\|$. Then
\begin{align*}
T_p(\x_k,\s_c)-T_p(\x_k,\mathbf{0})
=
\textbf{g}_k^\top \s_c
+
\frac12 \textbf{H}_k[\s_c]^2
+
\sum_{j=3}^p \frac{1}{j!}\nabla_x^j f(\x_k)[\s_c]^j \ge
-\|\textbf{g}_k\|\,\|\s_c\|
+\frac{c}{2}\|\s_c\|^2
-\sum_{j=3}^p \frac{\Lambda_j}{j!}\|\s_c\|^j.
\end{align*}
Using $\|\s_c\|=4c^{-1}\|\textbf{g}_k\|$ and \eqref{bound for g Tp}, we deduce that
$
T_p(\x_k,\s_c)-T_p(\x_k,\mathbf{0})>0.
$
Hence the minimizer of $T_p(\x_k,\s)$ over $\Omega_k$ cannot lie on the boundary $\partial\Omega_k$, and must therefore lie in the interior of $\Omega_k$.
Since $T_p(\x_k,\s)$ is strictly convex in $\Omega_k$, it follows that this minimizer is unique. Denote it by $\s_k^*$. By definition,
$
T_p(\x_k,\s_k^*)\le T_p(\x_k,\mathbf{0})=f(\x_k),
$
and equality holds only if $\s_k^*=\mathbf{0}$, since otherwise strict convexity would be violated.
\end{proof}

\begin{theorem}
\label{Thm converge in one step Tp}
Assume that $\textbf{H}_k \succeq cI_n \succ 0$ and
\begin{equation}
\|\textbf{g}_k\|
\le
\frac{c}{4}
\min\left\{
\min_{3\le j\le p}
\left(
\frac{p\Lambda_j}{(j-2)!c}
\right)^{-\frac{1}{j-2}},
\,
\left(\frac{\epsilon}{L}\right)^{1/p}
\right\}.
\label{bound for g 2 Tp}
\end{equation}
Let $\s_k^*$ be the unique local minimizer of $T_p(\x_k,\s)$ in $\Omega_k$. {Then
$
\big\|\nabla_{\x} f(\x_k+\s_k^*)\big\| \le \epsilon.
$
Consequently, \Cref{HTM algo} terminates after the step in near-stationary convex regime.}
\end{theorem}

\begin{proof}
By Lemma~\ref{lemma local min exists Tp}, the point $\s_k^*$ is well defined and lies in $\Omega_k$. Since it is an interior minimizer of $T_p(\x_k,\s)$, it satisfies
$
\nabla_{\s}T_p(\x_k,\s_k^*)=0.
$
Therefore, using Assumption~\ref{assumption Liptz},

\begin{align*}
\big\|\nabla_{\x} f(\x_k+\s_k^*)\big\|
=
\big\|\nabla_{\x} f(\x_k+\s_k^*)-\nabla_{\s}T_p(\x_k,\s_k^*)\big\| \le
L\|\s_k^*\|^p
\le
L(4c^{-1}\|\textbf{g}_k\|)^p.
\end{align*}
{By \eqref{bound for g 2 Tp}, we have
$
L(4c^{-1}\|\textbf{g}_k\|)^p\le \epsilon.
$
Hence
$
\big\|\nabla_{\x} f(\x_k+\s_k^*)\big\| \le \epsilon,
$
as required.}
\end{proof}

\paragraph{Step 2: refinement to accuracy $\epsilon$.}
\Cref{Thm converge in one step Tp} requires $\|\mathbf{g}_k\|$ to satisfy the smallness condition~\eqref{bound for g 2 Tp}. The next corollary expresses this in terms of the regime exponents $a,b$ {and isolates the additional local condition needed in the cubic case}.

{
\begin{assumption}[Local convex refinement]
\label{ass:local-refinement}
Whenever $\epsilon\le\|\mathbf g_k\|\le\epsilon^a$ and
$\mathbf H_k\succeq\epsilon^bI$, the local conditions of
\Cref{lemma local min exists Tp,Thm converge in one step Tp} hold. For
$p=3$, once the model step has produced a point with gradient
$O(\epsilon)$ in that strongly convex neighborhood, a monotone Newton-type
local solver is well defined and reaches $\|\nabla f\|\le\epsilon$ in at
most $O(\log\log(1/\epsilon))$ further objective and derivative
evaluations.
\end{assumption}
}

\begin{corollary}[Refinement to accuracy $\epsilon$]
\label{cor:newton-refine}
Let $a=\frac{p+1}{2p}$ and $b=\frac a2=\frac{p+1}{4p}$, and suppose $\epsilon\le\|\mathbf{g}_k\|\le\epsilon^a$ and $\mathbf{H}_k\succeq cI_n\succeq\epsilon^b I_n$. Let $\s_k^*=\argmin_{\s\in\Omega_k}T_p(\x_k,\s)$. Then
\[
\|\nabla_{\x} f(\x_k+\s_k^*)\|
\le L\big(4c^{-1}\|\mathbf{g}_k\|\big)^p
\le L4^{p}\epsilon^{p(a-b)}
= L4^{p}\epsilon^{(p+1)/4}
= O(\epsilon).
\]
For $p\ge4$ we have $p(a-b)=\frac{p+1}{4}>1$, so $\|\nabla_{\x} f(\x_k+\s_k^*)\|\le\epsilon$ already after this single step {for sufficiently small $\epsilon$}. For $p=3$ the exponent equals $1$, giving $\|\nabla_{\x} f(\x_k+\s_k^*)\|=O(\epsilon)$; {under Assumption \Cref{ass:local-refinement}}, {a further $O(\log\log(1/\epsilon))$ refinement evaluations yield $\|\nabla_{\x} f\|\le\epsilon$. Thus the convex-regime cost is of strictly lower order than the polynomial terms in the global complexity bound}, and \Cref{HTM algo} terminates.
\end{corollary}

\begin{proof}
The first chain is \Cref{Thm converge in one step Tp} together with $\|\s_k^*\|\le4c^{-1}\|\mathbf{g}_k\|\le4\epsilon^{a-b}$ and $c\ge\epsilon^b$; substituting $a=\frac{p+1}{2p}$, $b=\frac{p+1}{4p}$ gives $p(a-b)=\frac{p+1}{4}$. {The $p\ge4$ conclusion follows because the exponent is strictly larger than one. For $p=3$, the displayed estimate gives only $O(\epsilon)$, not necessarily $\epsilon$ with unit constant; the remaining statement is exactly \Cref{ass:local-refinement}, whose quadratically convergent local solver is available because $\nabla^2_\s T_p\succeq\frac{2c}{p}I_n\succ0$ on $\Omega_k$ by \Cref{lemma local min exists Tp}.}
\end{proof}

\begin{remark}
The Step-2 refinement is the trust-region analogue of the local acceleration used in AR$p$-type methods: once the iterate lies in a convex basin, asymptotic convergence is governed by the local rate rather than by the global model decrease. In practice the local solve can be a Newton step $\s_k=-\mathbf{H}_k^{-1}\mathbf{g}_k$, or a cubic/higher-order Newton step~\cite{silina2022unregularized}; relaxed refinement steps of this kind also appear in the QQR method~\cite{cartis2023second} and in ARC~\cite{cartis2007adaptive}, and more general convex high-order schemes are available~\cite{Nesterov2021implementable,Nesterov2020inexact,Nesterov2021inexact,Nesterov2021superfast,Nesterov2022quartic,Nesterov2006cubic}. Further local-convergence guarantees for $p$th-order methods are given in~\cite{welzel2025local}.
\end{remark}

{
\begin{theorem}[Overall conditional evaluation complexity]
\label{thm:overall-complexity}
Let $p\ge3$, $a=(p+1)/(2p)$, and $b=a/2$. Under
Assumptions~\ref{assumption Liptz}, \ref{assumption:bounded-derivatives}, \ref{ass:weak-curvature}, \ref{assumption epsilon small}, and \ref{ass:local-refinement}, the
reference safeguard \eqref{eq:safeguarded-boundary-step}, and the
initialization condition $\Delta_0\ge\kappa_0\epsilon^a$, Ada--HTM returns
an iterate satisfying $\|\nabla f(\x_k)\|\le\epsilon$ after
\[
O(\epsilon^{-2a})+O(\epsilon^{-3b})
+O(\log\log(1/\epsilon))
=O\!\left(\epsilon^{-\frac{p+1}{p}}\right)
\]
objective and derivative evaluations.
\end{theorem}
\begin{proof}
The far-field bound is \Cref{thm far complexity}. In either nonconvex
near-stationary regime, \Cref{thm:NSNC,thm:NSWC} gives a decrease of at least
$\kappa_{\rm ns}\epsilon^{3b}$ for a constant $\kappa_{\rm ns}>0$; since $f$ is
bounded below, there are at most $O(\epsilon^{-3b})$ such iterations. The
convex regime contributes the refinement cost in
\Cref{cor:newton-refine}. Finally, $2a=(p+1)/p$ and
$3b=3(p+1)/(4p)<2a$, so the far-field term dominates.
\end{proof}

The result is conditional on the displayed smoothness, nondegeneracy, and
local-refinement assumptions. It is not an unconditional improvement for
every high-order trust-region radius rule, and it does not require PAM to
solve the boundary problem globally.
}

\section{{Numerical Experiments}}
\label{sec numerical}


\subsection{Application I: Cubic polynomial optimization over the unit ball}

We first evaluate PAM (\Cref{alg1}) as an inner solver for the cubic polynomial
optimization problem with a unit-ball constraint (the CP-ball problem),
\begin{equation}\label{e19}
\tag{CP-ball}
\begin{aligned}
\min_{\x \in \mathbb{R}^n} &\hspace{0.4cm} f(\x)=\A\x^3+\x^\top B\x+\c^\top \x\\
\textup{s.t.}~ &\hspace{0.4cm}\|\x\|\leq1.
\end{aligned}
\end{equation}
{Unlike its quadratic counterpart, the trust-region subproblem, problem~\eqref{e19} is
NP-hard~\cite{buchheim2021lower}. It therefore provides a useful test of whether PAM can
attain solutions that are globally certified a posteriori on structured instances.}

{Introducing a slack variable $\theta$ with $\|\x\|^2+\theta^2=1$, followed by the
quadratic shift $-\alpha(\|\x\|^2+\theta^2)$, converts the CP-ball problem into the
sphere-constrained cubic model of \Cref{sec:CSPOP} in dimension $n+1$. We therefore solve
its multilinear reformulation using PAM. The complete reformulation, the sufficient shift
condition $\alpha\geq 3\|\A\|_F+\|B\|_F$ together with the resulting exactness guarantee
(\Cref{them5}), and the parameter-sensitivity results are given in \Cref{app:cpball}.}

\subsubsection{CP-ball: numerical results}\label{sec:exp-cspop}

\paragraph{Setup.} {All methods are implemented in Python (NumPy with Numba
just-in-time compilation) and run on a single CPU core. PAM uses
$\alpha=3\sqrt2\,\|\T\|_F$, proximal weights $\beta_i=1$, inner tolerance
$\varepsilon_{\rm in}=10^{-9}$, and a cap of $2000$ inner iterations. Results are
averaged over $20$ random instances from each of the five structural classes of the data
$(\g,\H,\T_0)$ in \Cref{tab:cspop-structures}. Parameter sensitivity is reported in
\Cref{app:cpball}. An open-source implementation of PAM and Ada--HTM, together with
scripts and fixed seeds reproducing all experiments in this section, is available at
\texttt{[repository URL to be inserted]}.}

\paragraph{PAM attains the certified global minimum.} {For dimensions at which the
order-$2$ moment--SOS (Lasserre) relaxation~\cite{lasserre2001global} is tractable, we
use its lower bound $L_{\rm sos}$ to certify the PAM value $U$: if
$U-L_{\rm sos}\le 10^{-3}$, {the PAM solution is SOS-certified to this
tolerance (the numerical SDP bound is not an exact symbolic certificate)}. At $n=5$,
PAM is certified on all $20$ instances in each of the five structural classes, with
median gaps of order $10^{-8}$ (\Cref{tab:cspop-structures}); this includes the
trust-region hard case (class~2) and the negative-curvature class (class~3), which
contains many competing local minima. These certificates are a posteriori: they
demonstrate the observed global accuracy of PAM on the tested instances, not a
worst-case guarantee, and the relaxation was tight on every instance for which we
computed it. Forming the certificate requires an SDP whose moment matrix has size
$\binom{n+2}{2}$ and whose cost grows as $O(n^6)$, so it is intractable beyond
$n\approx 10$. On the dimensions where the SDP is tractable, PAM is $40$--$245$ times
faster while matching its bound to approximately $10^{-8}$, and PAM can then be run far
beyond the range accessible to the certificate
(\Cref{fig:cspop-global}(a)--(b); timing details in \Cref{app:cpball}).}

\begin{table}[!htbp]
\centering
\caption{{Five CP-ball test classes at $n=5$ (data $(\g,\H,\T_0)$ with i.i.d.\ standard
normal entries). The final two columns report the number of PAM solutions certified by the
order-$2$ SOS relaxation and the median certificate gap over $20$ instances.}}
\label{tab:cspop-structures}
\begin{tabular}{llcc}
\hline
Class $(\g,\H,\T_0)$ & Character & cert.\ global & median $|U-L_{\rm sos}|$ \\
\hline
1. Dense $\T_0$, indefinite $\H$  & nonconvex (generic) & $20/20$ & $5.0\times10^{-8}$ \\
2. Dense $\T_0$, $\H\succ0,\ \g\approx\0$ & trust--region \emph{hard case} & $20/20$ & $2.6\times10^{-8}$ \\
3. Dense $\T_0$, neg. def. $\H$  & nonconvex (negative curvature) & $20/20$ & $4.8\times10^{-8}$ \\
4. $\H$ with eig.\ in $[10^{-6},1]$ & ill--conditioned ($\kappa\!\sim\!10^6$) & $20/20$ & $1.8\times10^{-8}$ \\
5. $\T_0$ rank--$1$ & low--rank cubic & $20/20$ & $4.3\times10^{-8}$ \\
\hline
\end{tabular}
\end{table}

\begin{figure}[!htbp]
\centering
\includegraphics[width=\textwidth]{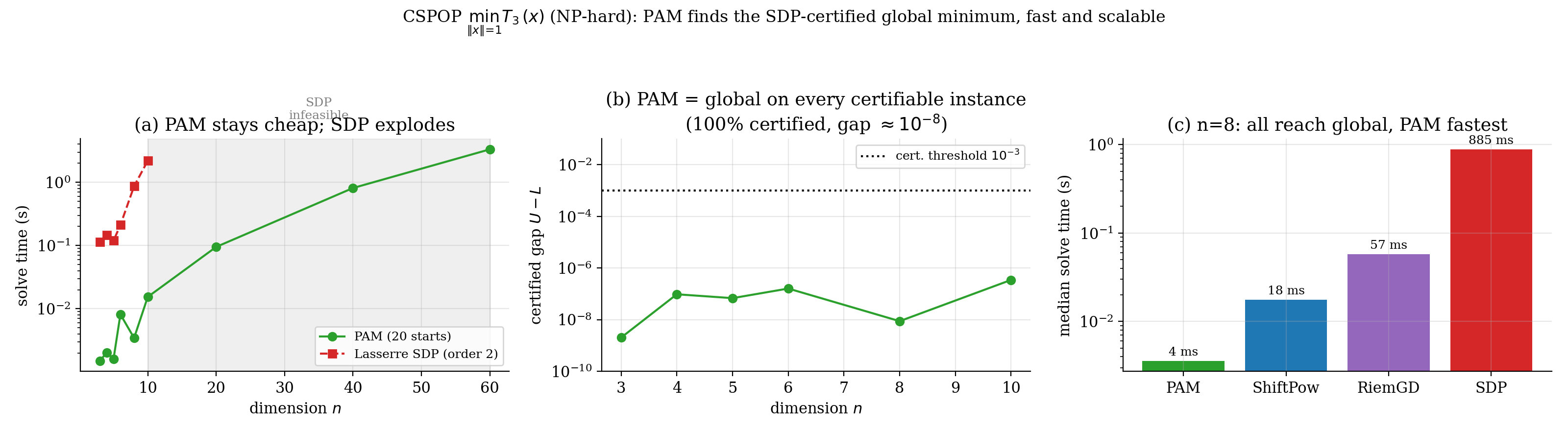}
\caption{{CSPOP $\min_{\|\x\|=1}T_3(\x)$, $20$ random instances per $n$ across the five
classes of \Cref{tab:cspop-structures}. (a) Wall time of PAM versus the order-$2$ SOS
certificate, which becomes intractable beyond $n\approx10$. (b) Certified gap
$U-L_{\rm sos}$ on the dimensions where the SDP is solvable. (c) Wall-time comparison of
sphere solvers at $n=8$.}}
\label{fig:cspop-global}
\end{figure}

\paragraph{Comparison with other sphere solvers.} {At $n=8$, with a small
multistart, PAM, Riemannian gradient descent~\cite{absil2008optimization}, and the
GEAP-style shifted-power method~\cite{kolda2014adaptive} all attain the certified value;
from a single random start, each solver reaches it on $16$ of the $20$ instances
(\Cref{tab:cspop-rest}). PAM is $4$--$14$ times faster
(\Cref{fig:cspop-global}(c)). Extended wall-time comparisons are reported
in \Cref{app:cpball} (\Cref{fig:cspop-walltime}).}

\paragraph{Low-rank tensors: structure-aware PAM scales.} When the
cubic tensor $\T_0$ is low-rank (class 5), {among the tested implementations
only the structure-aware PAM implementation exploits the rank-$R$ representation, with
block contractions costing $O(Rn)$ rather than $O(n^3)$; the shifted-power and Riemannian
baselines are run with dense contractions, so the comparison measures the benefit of a
structure-aware implementation rather than an intrinsic limitation of those methods}. On
a rank-$1$ instance at $n=120$, structure-aware PAM takes $1.7$\,ms, compared with
$94$\,ms for {dense} shifted power and $233$\,ms for dense PAM; all three
return {the same objective value}. Within a $10$\,s budget PAM reaches
$n\approx2386$, versus $n\approx154$--$296$ for the dense-contraction solvers
(\Cref{fig:cspop-lowrank}). {All contractions are evaluated matrix-free; the
closed-form formulas and their structured complexity are given in \Cref{app:cpball}
(\Cref{rem:closed-form-contraction}).}

\begin{figure}[h]
\centering
\includegraphics[width=0.6\textwidth]{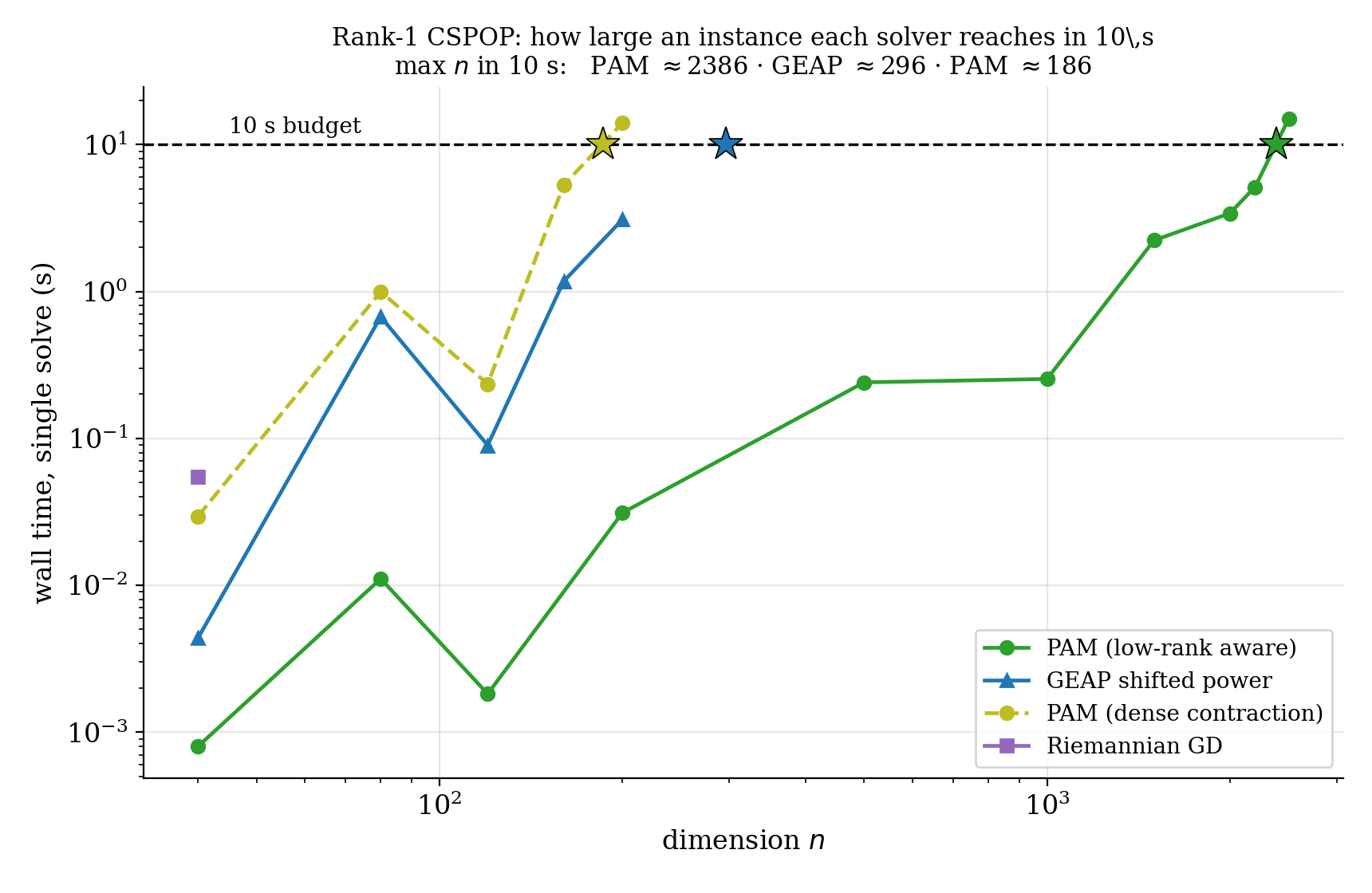}
\caption{{Rank-$1$ CSPOP: single-solve wall time versus dimension $n$ for
each sphere solver, with a $10$\,s budget (dashed) and the maximum $n$ reachable within
it (stars). Exploiting the rank-$1$ structure, PAM reaches $n\approx 2386$, against
$n\approx 154$--$296$ for the dense-contraction solvers.}}
\label{fig:cspop-lowrank}
\end{figure}

\subsection{The practical Ada--HTM method}

{Our implementation, which we call the \emph{practical Ada--HTM variant}, follows
\Cref{HTM algo} with one modified line. In the analysis, the convex regime is
triggered by $\lambda_{\min}(\mathbf H_k)\ge\epsilon^b$; in practice this trigger
misclassifies weakly convex valley floors ($0<\lambda_{\min}(\mathbf H_k)<\epsilon^b$)
as nonconvex, so we relax it to $\lambda_{\min}(\mathbf H_k)\ge0$. A near-singular
Hessian is safeguarded by the modified-Hessian shift
$\mathbf H_k+\max\{0,\tau-\lambda_{\min}(\mathbf H_k)\}\mathbf I$ with
$\tau=10^{-8}$, which keeps the shifted Hessian positive definite with margin
$\tau$. All three regimes (far-field, near-stationary nonconvex, convex) are
retained in the implementation. We note that \Cref{thm:overall-complexity}
applies to \Cref{HTM algo} with the trigger as analyzed; it does not directly
cover the relaxed practical trigger.}

\paragraph{An example for high-order trust-region: the Chebyshev--Rosenbrock problem}
\label{sec:exp-cheb}

We now evaluate the outer method, Ada--HTM (\Cref{HTM algo}) with the
CSPOP subproblem solved by PAM on a classical hard
nonconvex test problem by Nesterov.

The Chebyshev--Rosenbrock function was introduced by
Nesterov~\cite{jarre2011nesterov,nesterov2008private} and it is defined by
{
\[
f_r(\x)
= 0.25 (x_1 - 1)^2
  + \sum_{i=1}^{n-1} (x_{i+1} - 2x_i^2 + 1)^2 ,
\]
}
a fourth-degree polynomial with a global minimizer $x_{\text{glob}}^\ast = (1,\ldots,1)^T$ with $f_r(x^\ast)=0$.
At any other $x\in\mathbb{R}^n$, the gradient {$\nabla f_r(\x)$} is nonzero. The Chebyshev--Rosenbrock function is a challenging benchmark for nonlinear optimization.
Although the function has a unique stationary point (i.e., the global minimizer), the function contains many
\emph{near-stationary} regions with very small gradients, leading to slow
convergence for first- and second-order methods.
Classical methods such as Newton's method and BFGS tend to follow a highly oscillatory manifold before reaching the minimizer~\cite{gurbuzbalaban2012nesterov}.
This is precisely the setting in which third-order information matters.

We compare the practical \texttt{Ada-HTM} with gradient descent (GD), Newton's
method, \texttt{ARC}, and an exact second-order trust-region method (\texttt{TR}) from
the hard start $\x_0=(-1,1,\dots,1)^\top$.  The stopping tolerance is
$\epsilon_g=10^{-7}$ for $n\le8$, $10^{-9}$ for $n=9$, and $10^{-10}$ for $n\ge10$;
all methods use matched \textsf{Python} implementations{, with per-dimension
outer-iteration budgets shared by ARC, TR, and Ada--HTM ($8{,}000$ for $n\le8$, rising to
$1.3\times10^{5}$ at $n=12$) and fixed budgets of $20{,}000$ and $5{,}000$ iterations for
gradient descent and Newton's method, respectively; all runs use a fixed random seed}.

\begin{figure}[!htbp]
\centering
\includegraphics[width=\textwidth]{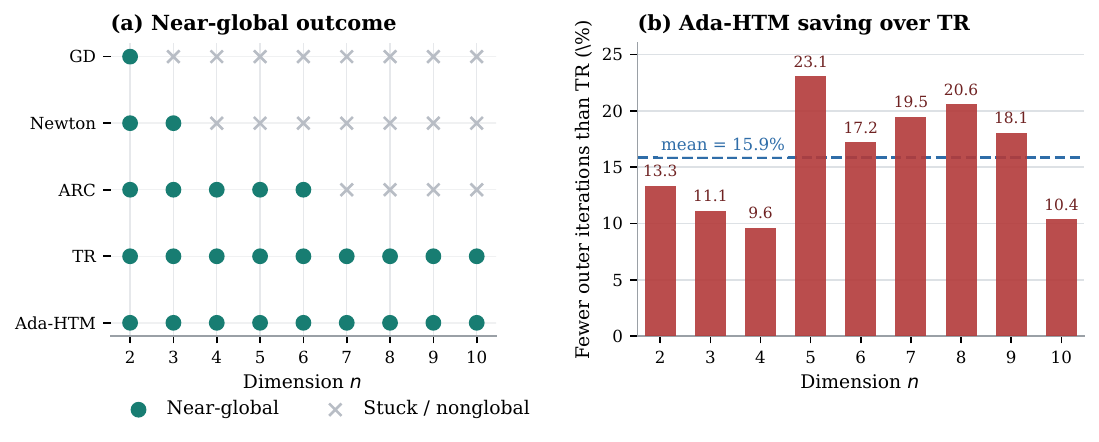}
\caption{{Chebyshev--Rosenbrock from the hard start. Left: whether
each method reaches the near-global basin
($\|\x^k-\x^\ast_{\rm glob}\|<5\times10^{-2}$) for $n=2,\ldots,10$.
Right: percentage reduction in outer iterations of Ada--HTM relative to TR,
$100(N_{\rm TR}-N_{\rm Ada})/N_{\rm TR}$; the mean saving is $15.9\%$.}}
\label{fig:cheb-summary}
\end{figure}

{
The left panel of \Cref{fig:cheb-summary} is deliberately an outcome plot:
it exposes the robustness difference more clearly than the full numerical
table. GD reaches the near-global basin only for $n=2$, Newton for $n=2,3$,
and ARC for $n=2,\ldots,6$. For $n=7,8,9,10$, ARC stops outside that basin
after $2391$, $5671$, $12667$, and $28289$ outer iterations, respectively.
TR and Ada--HTM are the only methods that reach the near-global basin at
every tested dimension $n=2,\ldots,10$. Because the global minimizer is the
only exact stationary point of this test, we do not interpret the failed
runs as convergence to another saddle or local minimum. Numerically, they
stop or exhaust the budget on a flat, weak-curvature portion of the curved
valley where the gradient is small but nonzero.

Among the two robust methods, Ada--HTM uses fewer outer iterations at every
dimension. The saving ranges from $9.6\%$ to $23.1\%$ and averages $15.9\%$;
at $n=9$ the counts are $11097$ versus $13546$, and at $n=10$ they are
$33423$ versus $37289$. The early trajectories in \Cref{fig:cheb-early}
show the mechanism. When both the gradient and quadratic curvature carry
little directional information, the cubic boundary model identifies a
descent direction and the far-field radius permits a larger accepted move
across the plateau. Ada--HTM therefore enters the Rosenbrock valley before
TR, after which its convex Newton phase follows the valley efficiently.

Wall-clock times are broadly comparable. A TR iteration is cheaper because
it uses only a quadratic model, whereas Ada--HTM pays for third-order
contractions and PAM sweeps. On this ill-conditioned valley, however, the
reduction in outer iterations largely offsets the extra work. For $n\ge11$,
all tested methods become trajectory-sensitive or fail to reach a reliable
near-global solution within the common budget; we therefore make no
scalability claim beyond $n=10$.
}
The full distances, iteration counts, and CPU times are reported in
\Cref{app:cheb-details}.

\begin{remark}[The assumptions in practice]
\label{rem:assumptions-numerics}
Assumptions~\ref{assumption Liptz}--\ref{assumption:bounded-derivatives} are standard smoothness and boundedness conditions
in evaluation-complexity analyses. They are not checked a priori when
running Ada--HTM. Along the trajectories generated in our experiments,
the relevant derivative operator norms remained finite. Except for the
deliberately badly scaled MGH instances, the largest observed derivative
and tensor norms were of order at most \(10^6\). The PAM-based CSPOP solver
remained numerically stable on the badly scaled instances, and these are
among the problems on which Ada--HTM performed particularly well.
\Cref{ass:weak-curvature} is needed only when an iteration enters the narrow
weak-curvature regime. An a posteriori examination of such iterations
showed that the cubic contraction along the signed reference direction
was nonzero, so no violation of the assumed nondegeneracy was observed.
\Cref{assumption epsilon small} is a technical small-tolerance condition that holds
automatically for all sufficiently small \(\epsilon\); its displayed
inequalities were also satisfied a posteriori for the tolerances and
constants observed in our experiments. The local-refinement condition
of \Cref{ass:local-refinement} is relevant only after entering a locally
positive-curvature region, where the observed refinement steps exhibited
the expected rapid convergence.
The reported implementation uses the candidate returned directly by PAM
and does not test these assumptions before taking a step. It also does not
apply the reference-step comparison in~\eqref{eq:safeguarded-boundary-step}.
That comparison is introduced to guarantee the model-decrease inequality
used in the outer complexity proof without requiring a globally optimal
solution of the boundary subproblem. Thus, the experiments provide
a posteriori evidence that the assumptions are compatible with the tested
problems, but do not remove the conditional nature of
\Cref{thm:overall-complexity}.
\end{remark}

\subsection{Robustness across the Mor\'e--Garbow--Hillstrom test set}
\label{sec:exp-mgh}

To test whether the previous behaviour extends beyond one tailored example, we run
Ada--HTM unchanged on $34$ Mor\'e--Garbow--Hillstrom (MGH) problems of dimension
$n=2$--$12$~\cite{more1981testing,birgin2018fortran}.  All methods use the standard
starting point, stop at $\|\nabla f\|_2\le10^{-5}$, and share a budget of $2000$ outer
iterations{; Ada--HTM's inner PAM solve uses three random restarts with a
fixed seed (ARC and TR employ their standard deterministic subproblem solvers)}.  We compare Ada--HTM with \texttt{ARC} and exact
\texttt{TR} using Dolan--Mor\'e profiles of outer iterations and CPU time
~\cite{dolan2002benchmarking,gould2016note,cartis2023second,zhu2024global}; problem~\#11
is excluded because its nonsmooth term has no third derivative.

\begin{figure}[!htbp]
\centering
\includegraphics[width=\textwidth]{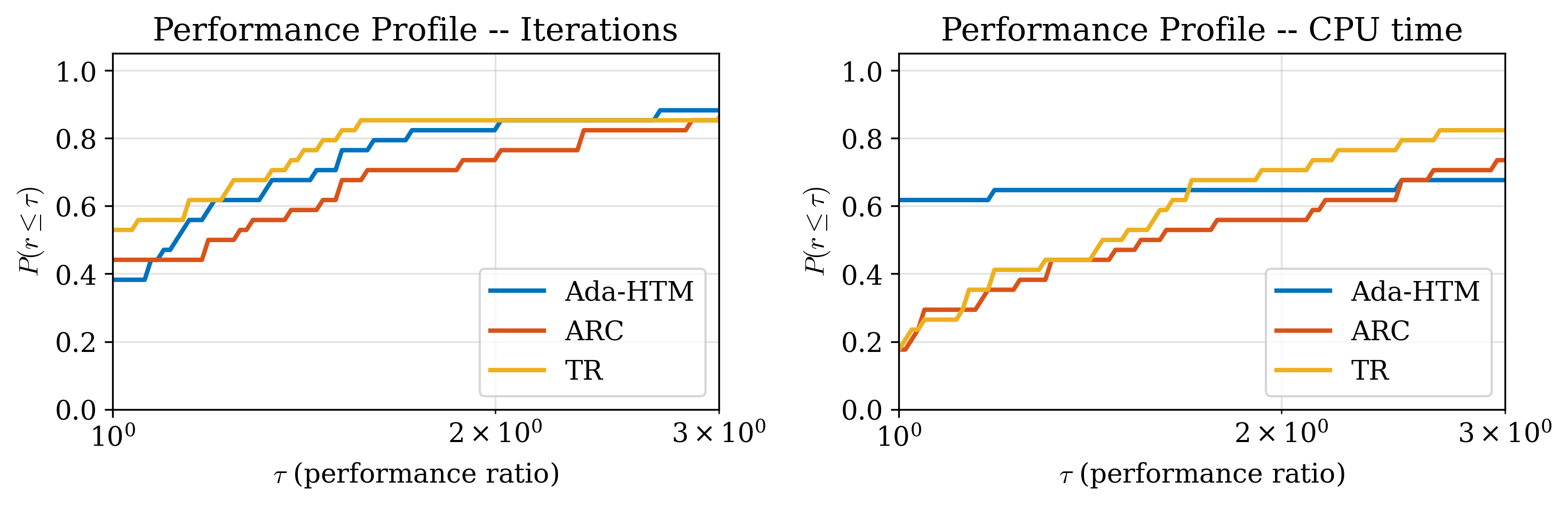}
\caption{Performance profiles on the $34$-problem MGH set for outer iterations
(left) and CPU time (right). All methods solve $32/34$ problems. Ada--HTM has the
highest iteration profile for $\tau\gtrsim2.5$ and is fastest on $62\%$ of the
problems at $\tau=1$; per-problem results are in \Cref{app:mgh}.}
\label{fig:mgh-profile}
\end{figure}

\begin{figure}[!htbp]
\centering
\includegraphics[width=.94\textwidth]{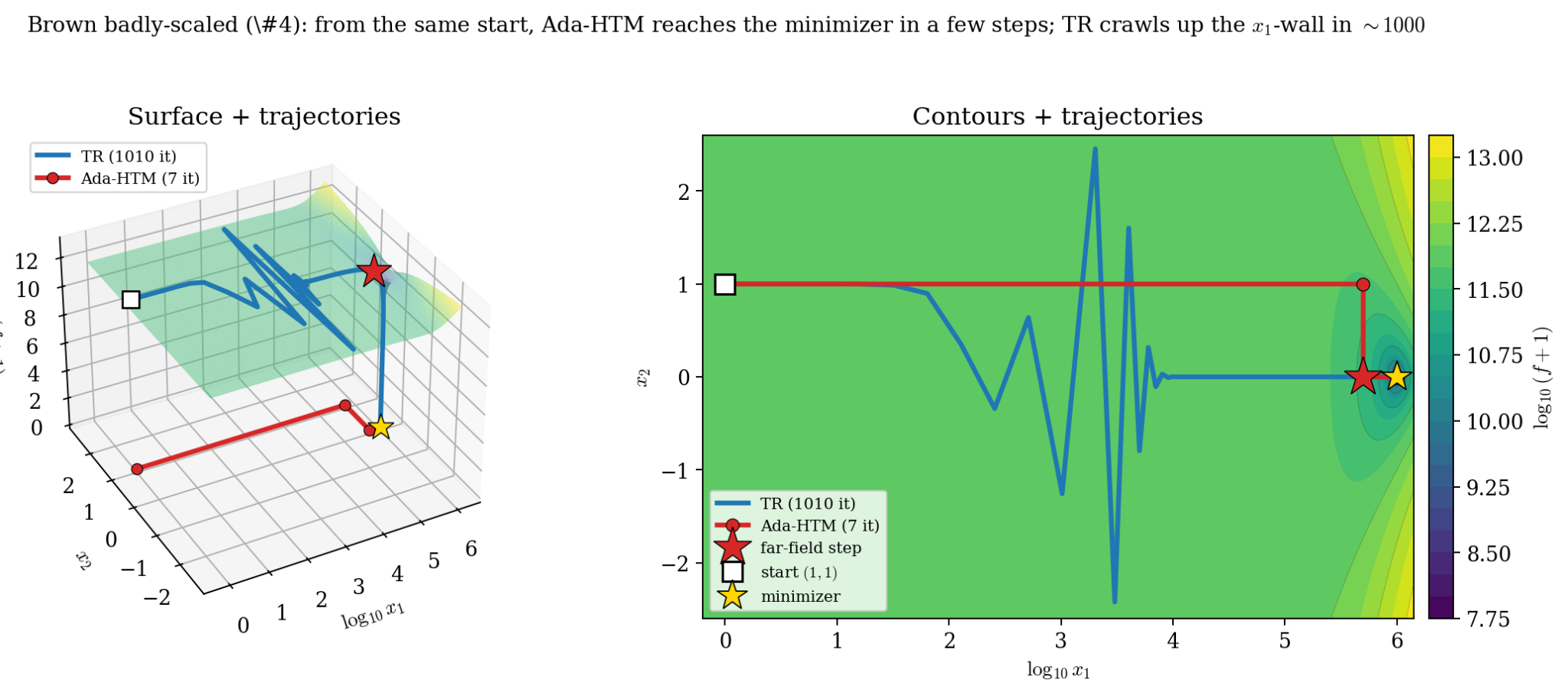}
\caption{{Brown badly-scaled (problem \#4) from the common start.
Ada--HTM reaches the minimizer in $7$ iterations, whereas TR requires $1010$
(ARC requires $1420$). The star marks the far-field step that corrects the
six-order-of-magnitude coordinate imbalance.}}
\label{fig:brown}
\end{figure}

The aggregate profiles show that the Chebyshev--Rosenbrock behaviour is not an isolated
effect.  All three methods solve $32/34$ problems, but Ada--HTM is especially effective on
badly-scaled and penalty-type objectives.  The clearest example is Brown badly-scaled
(\#4): Ada--HTM converges in $7$ iterations, compared with $1420$ for \texttt{ARC} and
$1010$ for \texttt{TR} --- {one far-field third-order step corrects the $10^6$ coordinate
mismatch, after which four convex Newton steps finish the solve (\Cref{fig:brown})}.  It also reduces Powell badly-scaled (\#3) from $1379$
\texttt{ARC} iterations to $93$, and succeeds on Brown--Dennis (\#16), where \texttt{TR}
stops above the requested tolerance.

{
On CPU time the methods are broadly comparable, with Ada--HTM fastest on a
majority of the set. Its advantage is driven mainly by ill-conditioned,
badly-scaled, and penalty-type objectives, where a small number of
informative third-order steps replaces hundreds or thousands of
radius-limited second-order iterations. The benefit is not uniform: when
third derivatives are obtained by finite differences, their cost can offset
an iteration reduction. Thus Ada--HTM is most valuable on the difficult
geometry that motivates the method, while remaining computationally
competitive on the rest of the set.}
The complete per-problem results are given in \Cref{app:mgh}.

\section{Conclusion}
\label{sec conclusion}

{
We developed a homogeneous-tensor route to radius-controlled high-order
Taylor steps. The representation is exact for arbitrary $p$ and turns the
model on a prescribed boundary into a spherical polynomial problem without
requiring the augmented tensor to be assembled. For the implemented cubic
case, an explicit quadratic shift gives global equivalence between the
inhomogeneous cubic problem and a three-block multilinear formulation. This
equivalence motivates a PAM solver with closed-form block updates. Its
guarantee is deliberately local: objective values decrease and accumulation
points are stationary; globality is asserted only with an independent
certificate.

Embedding these boundary candidates in Ada--HTM gives an outer framework
whose analysis separates far-field progress from near-stationary curvature
regimes. The reference safeguard is the critical inner--outer interface:
PAM proposes a step, while an explicit gradient or curvature direction
guarantees the comparison used in the proofs. Under the stated smoothness,
safeguarded-decrease, weak-curvature nondegeneracy, and local-refinement
conditions, Ada--HTM matches the AR$p$ evaluation-complexity order
$O(\epsilon^{-(p+1)/p})$.

The experiments support this two-level interpretation. PAM agrees with
order-$2$ moment--SOS certificates where certification is tractable and is
especially effective for low-rank tensors. In the outer tests, Ada--HTM is
competitive with exact trust-region and cubic-regularization methods, with
its clearest advantages on weak-curvature, ill-conditioned, and badly-scaled
problems. Extending the practical solver beyond $p=3$, relating PAM residuals
more sharply to outer stopping rules, and exploiting sparse and low-rank
tensor-free contractions are the most immediate next steps.
}

\medskip

\medskip

\noindent{\bf Acknowledgements:} This work was supported by the National Natural Science Foundation of China (12071249), Shandong Provincial Natural Science Foundation ZR2021JQ01,ZR2024MA003). 
The research of Wenqi Zhu was also supported by the Engineering and Physical Sciences
Research Council under grant EP/Y028872/1, \emph{Mathematical Foundations
of Intelligence: An ``Erlangen Programme'' for AI}.

\appendix

\section{Explicit Homogeneous Formulation}
\label{sec example}

\begin{remark} \textbf{Homogeneous formulation for third order Taylor polynomial:} Denote $f_0$, $\textbf g_k$, $\textbf H_k$, and $\textbf T_k$ for $f(\x_k) $ $\nabla f(\x_k)$, $\nabla^2 f(\x_k)$, and $\nabla^3 f(\x_k)$, respectively. Then, \begin{eqnarray} \label{third order taylor} T_3(\s) = f_0 + \textbf g_k^\top \s +\frac{1}{2}\textbf H_k[\s]^2 +\frac{1}{6}\textbf T_k[\s]^3 = \A[\tilde{\s}]^3, \end{eqnarray} where $\A\in\R^{(n+1)^3}$ is a third-order symmetric tensor and $\tilde{\s}=(1,\s^\top)^\top\in\mathbb R^{n+1}$. Specifically, using \eqref{formulation for pth order}, the nonzero entries of $\A$ are given by entries of $\textbf g_k$, $\textbf H_k$, and $\textbf T_k$ \[ \A_{000}= f_0,\qquad \A_{\pi(i00)}=\frac{1}{3}[\textbf g_k]_i,\qquad \A_{\pi(ij0)}=\frac{1}{6}[\textbf H_k]_{ij}, \] \[ {\A_{\pi(ij\iota)}=\frac{1}{6}[\textbf T_k]_{ij\iota}, \qquad i,j,\iota\in[n].} \]
{(The tensor contraction $\A[\tilde{\s}]^3$ already sums over all index permutations; since $\textbf T_k$ is symmetric, the cubic entries carry the uniform factor $1/6$ for every index pattern, in agreement with the general formula \eqref{formulation for pth order}.)} All other entries are zero, and $\pi$ denotes permutations of the indices. An explicit example is given at the end of this appendix. \label{remark Homogeneous formulation for third order Taylor polynomial} \end{remark} \begin{remark} \textbf{Homogeneous formulation for the AR$3$ subproblem.} If $p=3$ in \eqref{subprob}, then $ m_3(\s) = \textbf g_k^\top\s+\frac{1}{2}\textbf H_k[\s]^2+\frac{1}{6}\textbf T_k[\s]^3+\frac{\sigma}{4}\|\s\|^4. $ Define a fourth-order symmetric tensor $\M\in\R^{(n+1)^4}$ using \eqref{formulation for pth order}, the nonzero entries of $\M$ are given as follows: \[ \M_{0000}=f_0, \qquad \M_{\pi(i000)}=\frac14 g_i, \qquad {\M_{\pi(ij00)}=\frac{1}{12}H_{ij}\quad (i,j\in[n]),} \] \[ {\M_{\pi(ij\iota 0)}=\frac{1}{24}T_{ij\iota}\quad (i,j,\iota\in[n]),} \] and for the quartic regularization term, $ \M_{\pi(i i j j)}=\frac{\sigma}{12}\text{ } (i\neq j), \text{ } \M_{iiii}=\frac{\sigma}{4}. $ All other entries are zero. \end{remark}

\begin{example}
Suppose that $T_3(\x)\in\mathbb{R}_3[\x]$ is the polynomial
\[
T_3(\x)=x_1^3+x_2^3+x_3^3+x_1x_2+x_2x_3+x_1+x_2-3.
\]
Let $\T\in\mathbb{R}^{4\times4\times4}$ be a third-order symmetric tensor and
$\tilde{\x}=(1,\x^\top)^\top\in\mathbb{R}^{4}$.
By the construction in \eqref{formulation for pth order}, the tensor $\T$ has the following entries:
\[
\T_{0,:,:}=
\begin{pmatrix}
-3 & \frac{1}{3} & \frac{1}{3} & 0\\[4pt]
\frac{1}{3} & 0 & \frac{1}{6} & 0\\[4pt]
\frac{1}{3} & \frac{1}{6} & 0 & \frac{1}{6}\\[4pt]
0 & 0 & \frac{1}{6} & 0
\end{pmatrix},
\quad
\T_{1,:,:}=
\begin{pmatrix}
\frac{1}{3} & 0 & \frac{1}{6} & 0\\[4pt]
0 & 1 & 0 & 0\\[4pt]
\frac{1}{6} & 0 & 0 & 0\\[4pt]
0 & 0 & 0 & 0
\end{pmatrix},
\quad
\T_{2,:,:}=
\begin{pmatrix}
\frac{1}{3} & \frac{1}{6} & 0 & \frac{1}{6}\\[4pt]
\frac{1}{6} & 0 & 0 & 0\\[4pt]
0 & 0 & 1 & 0\\[4pt]
\frac{1}{6} & 0 & 0 & 0
\end{pmatrix},
\quad
\T_{3,:,:}=
\begin{pmatrix}
0 & 0 & \frac{1}{6} & 0\\[4pt]
0 & 0 & 0 & 0\\[4pt]
\frac{1}{6} & 0 & 0 & 0\\[4pt]
0 & 0 & 0 & 1
\end{pmatrix}.
\]
By direct computation, we obtain
$
T_3(\x)
=
\T[\tilde{\x}]^3
=
\sum_{i_1,i_2,i_3\in\{0,1,2,3\}}
\T_{i_1 i_2 i_3}
\tilde{x}_{i_1}\tilde{x}_{i_2}\tilde{x}_{i_3}.
$
\end{example}
\begin{remark}[Equivalence of the order-3 and order-4 formulations]
The cubic Taylor model admits two exact homogenisations: as an order-3 form
$T_3(s)=\A[\tilde s]^3$ with $\A\in\R^{(n+1)^3}$, and (after quartic regularisation)
as an order-4 form $m_3(s)=T_3(s)+\tfrac{\sigma}{4}\|s\|^4=\M[\tilde s]^4$ with
$\M\in\R^{(n+1)^4}$. Both identities hold for every $s$, so optimising the
homogeneous form is identical to optimising the model; the order is purely a
representational choice. By the maximum-block-improvement argument of
Theorems~\ref{thm1}--\ref{them2}, the associated multilinear (CSMOP) problems are
equivalent tightenings at either order for sufficiently large $\alpha$, and the
corresponding $3$-block and $4$-block PAM solvers agree to machine precision on our
tests. We therefore present the order-3 formulation without loss of generality; the
order-4 (even) form is convenient only when an even-degree tensor is required, e.g.\
for $Z$-eigenvalue/shifted-power solvers or for the SOS relaxation of the regularised
subproblem.
\end{remark}

\section{Additional Chebyshev--Rosenbrock results}\label{app:cheb-details}

\Cref{tab:cheb-compare-full} gives the numerical values summarized by
\Cref{fig:cheb-summary}.  A dash marks a final distance greater than $0.5$; bold
entries satisfy $\|\x^k-\x^\ast_{\rm glob}\|\le10^{-2}$.  Outer iterations count
all subproblem solves, and CPU times use matched \textsf{Python} implementations.

\begin{table}[!htbp]
\centering
\caption{Chebyshev--Rosenbrock from $\x_0=(-1,1,\dots,1)^\top$: final distance,
outer iterations, and CPU time.}
\label{tab:cheb-compare-full}
\renewcommand{\arraystretch}{1.12}
\resizebox{\textwidth}{!}{
\begin{tabular}{c | c c c c c | c c c | c c c}
\hline
 & \multicolumn{5}{c|}{distance $\|\x^k-\x^\ast_{\rm glob}\|$}
 & \multicolumn{3}{c|}{outer iterations} & \multicolumn{3}{c}{CPU time (s)}\\
$n$ & GD & Newton & \texttt{ARC} & \texttt{TR} & \texttt{Ada} & \texttt{ARC} & \texttt{TR} & \texttt{Ada} & \texttt{ARC} & \texttt{TR} & \texttt{Ada}\\
\hline
2 & \textbf{3.2e-4} & \textbf{2.6e-8} & \textbf{6.7e-10} & \textbf{2.8e-11} & \textbf{2.6e-8} & 14 & 15 & 13 & $<$0.1 & $<$0.1 & $<$0.1 \\
3 & 5.4e-2 & \textbf{1.3e-5} & \textbf{1e-8} & \textbf{5.7e-9} & \textbf{1.9e-9} & 38 & 36 & 32 & $<$0.1 & $<$0.1 & $<$0.1 \\
4 & -- & -- & \textbf{4.3e-7} & \textbf{1.4e-7} & \textbf{5.5e-8} & 117 & 83 & 75 & $<$0.1 & $<$0.1 & 0.2 \\
5 & -- & -- & \textbf{8.1e-4} & \textbf{2.6e-6} & \textbf{3.6e-7} & 440 & 260 & 200 & 0.1 & $<$0.1 & $<$0.1 \\
6 & -- & -- & \textbf{7.8e-3} & \textbf{2.2e-5} & \textbf{3.9e-6} & 1218 & 639 & 529 & 0.2 & $<$0.1 & $<$0.1 \\
7 & -- & -- & -- & \textbf{5.8e-4} & \textbf{1.9e-3} & 2391 & 1772 & 1427 & 0.5 & 0.1 & 0.2 \\
8 & -- & -- & -- & \textbf{9.5e-3} & \textbf{3.9e-3} & 5671 & 5034 & 3998 & 1.0 & 0.3 & 0.4 \\
9 & -- & -- & -- & \textbf{2.4e-5} & \textbf{3e-3} & 12667 & 13546 & 11097 & 2.5 & 0.7 & 1.1 \\
10 & -- & -- & -- & \textbf{5.9e-5} &\textbf{ 1.5e-4} & 28289 & 37289 & 33423 & 5.9 & 2.4 & 2.9 \\
\hline
\end{tabular}}
\end{table}

\Cref{fig:cheb-early} resolves the two sources of the iteration advantage.  The
far-field third-order step gives Ada--HTM an early lead when leaving the plateau,
while its convex Newton phase continues to gain inside the valley.

\begin{figure}[!htbp]
\centering
\includegraphics[width=\textwidth]{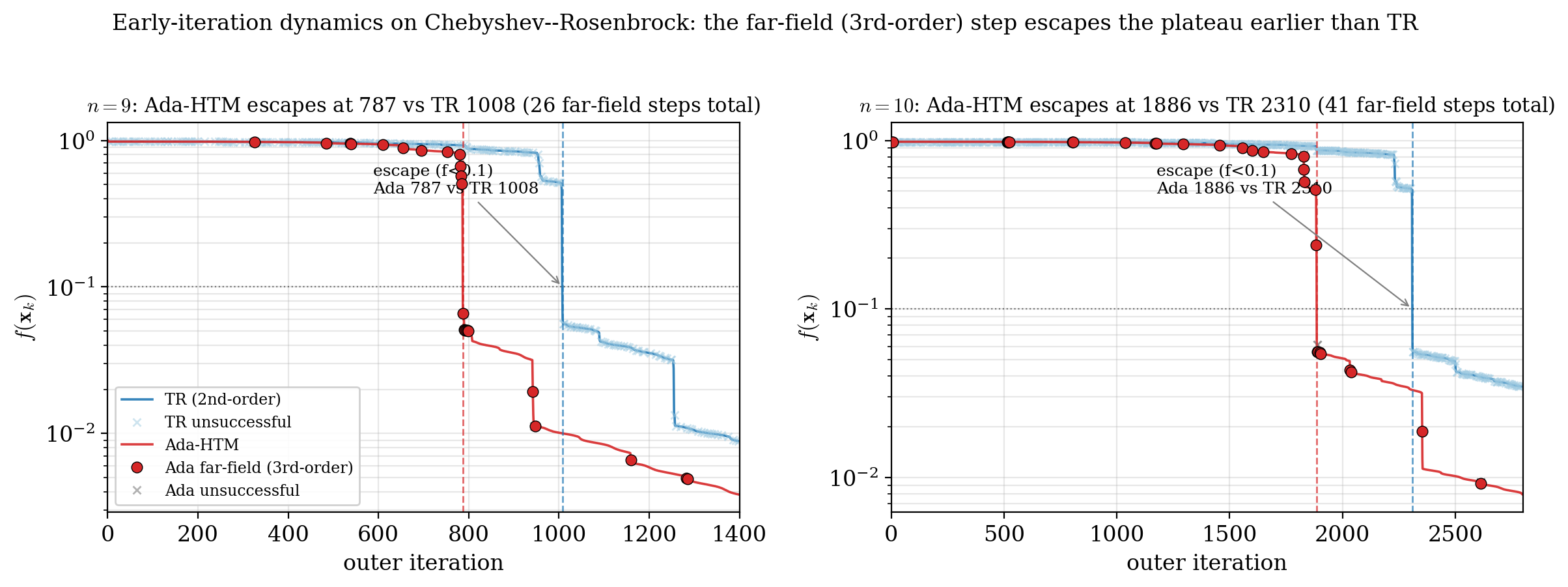}
\caption{Early Chebyshev--Rosenbrock dynamics for $n=9$ (left) and $n=10$
(right). Ada--HTM's far-field steps are circles and rejected \texttt{TR} trials are
crosses. Ada--HTM reaches $f<0.1$ at iterations $787$ and $1886$, compared with
$1008$ and $2310$ for \texttt{TR}.}
\label{fig:cheb-early}
\end{figure}

\section{Per-problem results on the MGH test set}\label{app:mgh}

\Cref{tab:mgh} reports, for every problem of the $34$-problem
Mor\'e--Garbow--Hillstrom set used in \Cref{sec:exp-mgh}, the number of outer iterations
(\emph{it}), the number of function/derivative evaluations (\emph{ev}), and the CPU time (s)
for Ada--HTM, \texttt{ARC} and \texttt{TR}, from the standard starting point with stopping
rule $\|\nabla f\|_2\le10^{-5}$ and a budget of $2000$ iterations. A superscript $^{\ast}$
marks a run that did not reach the tolerance (Meyer \#10 for all three; Osborne-1 \#17 for
Ada--HTM and \texttt{ARC}; Brown--Dennis \#16 for \texttt{TR}); each method solves $32/34$.
For problems \#19--35 the third derivative is finite-differenced, so the CPU of the
$\nabla^3 f$-using method (Ada--HTM) is inflated on a few of them (most visibly \#19 and
\#35); the iteration and evaluation counts are unaffected.

\begin{table}[!htbp]
\centering
\caption{Per-problem outer iterations (\emph{it}), function/derivative
evaluations (\emph{ev}) and CPU time (s) on the $34$-problem MGH set for Ada--HTM,
\texttt{ARC} and \texttt{TR}. A superscript $^{\ast}$ marks a run that exhausted the budget
without reaching $\|\nabla f\|_2\le10^{-5}$.}
\label{tab:mgh}
\resizebox{\textwidth}{!}{
\begin{tabular}{r l c | r r r | r r r | r r r}
\hline
 & & & \multicolumn{3}{c|}{Ada--HTM} & \multicolumn{3}{c|}{\texttt{ARC}} & \multicolumn{3}{c}{\texttt{TR}}\\
\# & problem & $n$ & it & ev & s & it & ev & s & it & ev & s\\
\hline
1 & Rosenbrock & 2 & 22 & 145 & $<$0.1 & 33 & 74 & $<$0.1 & 25 & 75 & $<$0.1 \\
2 & Freudenstein--Roth & 2 & 7 & 41 & $<$0.1 & 11 & 36 & $<$0.1 & 8 & 27 & $<$0.1 \\
3 & Powell badly-scaled & 2 & 93 & 597 & 0.3 & 1379 & 4130 & 0.3 & 114 & 329 & $<$0.1 \\
4 & Brown badly-scaled & 2 & 7 & 41 & $<$0.1 & 1420 & 4263 & 0.3 & 1010 & 3033 & 0.1 \\
5 & Beale & 2 & 12 & 71 & $<$0.1 & 9 & 24 & $<$0.1 & 7 & 23 & $<$0.1 \\
6 & Jennrich--Sampson & 2 & 10 & 59 & $<$0.1 & 9 & 30 & $<$0.1 & 9 & 30 & $<$0.1 \\
7 & Helical valley & 3 & 13 & 73 & $<$0.1 & 21 & 48 & $<$0.1 & 9 & 29 & $<$0.1 \\
8 & Bard & 3 & 9 & 53 & $<$0.1 & 9 & 30 & $<$0.1 & 14 & 45 & $<$0.1 \\
9 & Gaussian & 3 & 2 & 11 & $<$0.1 & 1 & 6 & $<$0.1 & 1 & 6 & $<$0.1 \\
10 & Meyer & 3 & 2000$^{\ast}$ & 149468 & 2.7 & 2000$^{\ast}$ & 5959 & 1.1 & 254$^{\ast}$ & 733 & 0.1 \\
12 & Box-3D & 3 & 16 & 95 & $<$0.1 & 32 & 89 & $<$0.1 & 16 & 51 & $<$0.1 \\
13 & Powell singular & 4 & 16 & 95 & $<$0.1 & 15 & 48 & $<$0.1 & 15 & 48 & $<$0.1 \\
14 & Wood & 4 & 43 & 261 & $<$0.1 & 51 & 134 & $<$0.1 & 43 & 126 & $<$0.1 \\
15 & Kowalik--Osborne & 4 & 12 & 68 & $<$0.1 & 8 & 23 & $<$0.1 & 8 & 25 & $<$0.1 \\
16 & Brown--Dennis & 4 & 9 & 53 & $<$0.1 & 8 & 27 & $<$0.1 & 10$^{\ast}$ & 35 & $<$0.1 \\
17 & Osborne-1 & 5 & 2000$^{\ast}$ & 12956 & 6.2 & 2000$^{\ast}$ & 5955 & 2.3 & 32 & 92 & $<$0.1 \\
18 & Biggs EXP6 & 6 & 70 & 383 & 0.2 & 26 & 77 & $<$0.1 & 38 & 105 & $<$0.1 \\
19 & Osborne-2 & 11 & 21 & 120 & 15.9 & 20 & 57 & 0.8 & 16 & 49 & 0.8 \\
20 & Watson & 6 & 12 & 71 & $<$0.1 & 16 & 51 & 0.1 & 11 & 36 & $<$0.1 \\
21 & Extended Rosenbrock & 10 & 22 & 145 & $<$0.1 & 26 & 71 & $<$0.1 & 23 & 70 & $<$0.1 \\
22 & Extended Powell singular & 12 & 16 & 95 & $<$0.1 & 16 & 51 & $<$0.1 & 16 & 51 & $<$0.1 \\
23 & Penalty I & 10 & 34 & 217 & $<$0.1 & 79 & 138 & $<$0.1 & 45 & 124 & $<$0.1 \\
24 & Penalty II & 10 & 90 & 607 & 0.1 & 255 & 382 & 0.2 & 111 & 315 & 0.2 \\
25 & Variably dimensioned & 10 & 15 & 89 & $<$0.1 & 14 & 45 & $<$0.1 & 14 & 45 & $<$0.1 \\
26 & Trigonometric & 10 & 8 & 46 & 0.8 & 11 & 32 & 0.1 & 11 & 34 & 0.1 \\
27 & Brown almost-linear & 10 & 8 & 47 & $<$0.1 & 5 & 18 & $<$0.1 & 7 & 24 & $<$0.1 \\
28 & Discrete boundary-value & 10 & 3 & 17 & $<$0.1 & 3 & 12 & $<$0.1 & 2 & 9 & $<$0.1 \\
29 & Discrete integral-eqn & 10 & 4 & 23 & $<$0.1 & 3 & 12 & $<$0.1 & 3 & 12 & $<$0.1 \\
30 & Broyden tridiagonal & 10 & 6 & 35 & $<$0.1 & 5 & 18 & $<$0.1 & 5 & 18 & $<$0.1 \\
31 & Broyden banded & 10 & 8 & 47 & $<$0.1 & 7 & 24 & $<$0.1 & 7 & 24 & $<$0.1 \\
32 & Linear (full rank) & 10 & 2 & 11 & $<$0.1 & 6 & 21 & $<$0.1 & 3 & 12 & $<$0.1 \\
33 & Linear (rank 1) & 10 & 25 & 124 & 0.3 & 2 & 9 & $<$0.1 & 6 & 19 & $<$0.1 \\
34 & Linear (rank 1, 0 cols) & 10 & 22 & 101 & 0.1 & 2 & 9 & $<$0.1 & 7 & 21 & $<$0.1 \\
35 & Chebyquad & 8 & 19 & 104 & 4.9 & 30 & 51 & 0.2 & 16 & 46 & 0.4 \\
\hline
\end{tabular}}
\end{table}

\begin{remark}[On the fairness of the comparison]\label{rem:fairness}
All methods use identical starting points, the same outer stopping tolerance
$\|\nabla f\|\le\epsilon_g$, and {the per-dimension iteration budgets reported in
\Cref{sec:exp-cheb}}. Two implementation
details are noted for reproducibility. {(i) In \texttt{Ada-HTM} the CSPOP subproblem
is \emph{warm-started} along $\pm v_{\min}(\mathbf H_k)$, the weakest-curvature
direction of the current Hessian; \texttt{ARC} and \texttt{TR} initialize their inner
solves from the current iterate in the standard way. The eigenvector used by the warm
start is computed from the same Hessian available to all methods, and the cost of
computing it is included in \texttt{Ada-HTM}'s reported CPU time; we note explicitly
that the initializations are not identical in information content.} (ii) The
reported iteration count (\texttt{nit}) is the number of \emph{outer} iterations;
inner-solver work (PAM sweeps, the moment-SDP, or the secular/trust-region solve) is
counted separately.
\end{remark}

\paragraph{Brown badly-scaled.}
For $f(\x)=(x_1-10^{6})^2+(x_2-2\times10^{-6})^2+(x_1x_2-2)^2$,
\Cref{fig:brown} {(in \Cref{sec:exp-mgh})} shows that one far-field step moves $x_2$ from $1$ to approximately
$2\times10^{-3}$, after which four convex Newton steps converge to the minimizer.

\section{Details for Application I (the CP-ball problem)}\label{app:cpball}

This appendix contains the complete reduction of the CP-ball problem \eqref{e19} to a
CSPOP, the supporting lemma and the proof of \Cref{them5}, the inner algorithm, the
parameter-sensitivity study, and the additional timing tables and large-scale wall-time
figure referenced from \Cref{sec:exp-cspop}.

\paragraph{Reduction to a CSPOP.}
{Introducing a slack variable $\theta$ with $\|\x\|^2+\theta^2=1$ makes \eqref{e19} equivalent to
the sphere-constrained problem} (the slack is written $\theta$ because $\t$ denotes the block triple of \Cref{sec:CSPOP})
\begin{equation}\label{e20}
\min_{\x \in \mathbb{R}^n,\,\theta\in\mathbb{R}}~ f(\x)=\A\x^3+\x^\top B\x+\c^\top \x
\quad \textup{s.t.}~ \|\x\|^2+\theta^2=1 .
\end{equation}
Writing $\v=(\theta,\x^\top)^\top\in\mathbb{R}^{n+1}$ and subtracting the constant
$\alpha(\|\x\|^2+\theta^2)=\alpha$, the shifted objective $h(\v)=f(\x)-\alpha\|\v\|^2$ turns
\eqref{e20} into the CSPOP
\begin{equation}\label{e22}
\min_{\v \in \mathbb{R}^{n+1}}~ h(\v)=\Big\langle\T_h,\,\tilde\v\circ\tilde\v\circ\tilde\v\Big\rangle
\quad \textup{s.t.}~ \|\v\|^2=1,
\end{equation}
where $\tilde\v=(1,\v^\top)^\top$ and $\T_h\in\mathbb{R}^{(n+2)^3}$ is the coefficient
tensor of $h$; this is exactly the CSPOP form~\eqref{e1 cspop} of
\Cref{sec:CSPOP}, in dimension $n+1$. Its associated multilinear problem (CSMOP) is
\begin{equation}\label{e23}
\min_{\u,\v,\w \in \mathbb{R}^{n+1}}~ F_h(\u,\v,\w)=\Big\langle\T_h,\,\tilde\u\circ\tilde\v\circ\tilde\w\Big\rangle
\quad \textup{s.t.}~ \|\u\|=\|\v\|=\|\w\|=1 .
\end{equation}

\begin{remark}[Solving the CP-ball problem with PAM]
Problem~\eqref{e19} is solved by applying the PAM scheme of \Cref{alg1} to the
CSMOP~\eqref{e23}; the only changes are that the cubic data are those of the shifted
polynomial $h$, and that the solution is read off the $\x$-block of the recovered
$\v$. The full procedure is stated as \Cref{alg:cpball} below, and is
otherwise identical to \Cref{alg1}.
\end{remark}

The PAM solver has two parameters: the shift $\alpha$, which makes the homogenisation
exact, and the proximal weights $\beta_i$, which regularise the block updates. A
sensitivity study over both parameters is reported in \Cref{fig:cspop-alpha-gamma} below;
the solver is robust across a wide range, and we use $\alpha=3\sqrt2\,\|\T\|_F$ and
$\beta_i=1$ throughout. In particular, \Cref{lema2} below shows that any
$\alpha\ge 3\|\A\|_F+\|B\|_F$ makes $\nabla^2 h$ negative semidefinite on the unit ball,
which guides this choice and yields the following equivalence, whose proof is given after
\Cref{lema2}.
\begin{theorem}\label{them5}
Let $\alpha\geq3\|\A\|_F+\|B\|_F$. Then \eqref{e22} and \eqref{e23} are equivalent, and
if $\u^*,\v^*,\w^*\in\mathbb{R}^{n+1}$ are optimal for \eqref{e23} then they are optimal
for \eqref{e22}, i.e. $h(\u^*)=h(\v^*)=h(\w^*)=F_h(\u^*,\v^*,\w^*)$.
\end{theorem}

\begin{lemma}\label{lema2}
Let $\alpha\geq3\|\A\|_F+\|B\|_F$. Then, {for any $\v=(\theta,\x^\top)^\top$ with $\|\x\|\leq 1$ and $\theta\in\mathbb{R}$, the Hessian $\nabla^2h(\v)$} is negative semi-definite.
\end{lemma}
\proof By a direct computation, we obtain that
$$
\nabla^2h(\v)=
\left(\begin{matrix}
-2\alpha & \0^\top \\
\0 & 6\A\x+2B-2\alpha I_n
\end{matrix}
\right).
$$
Since $\alpha\geq3\|\A\|_F+\|B\|_F$, for any $\bar{\v}=(\omega_0,\w^\top)^\top\in\mathbb{R}^{n+1}$ it holds that
$$
\begin{aligned}
{\bar{\v}^\top\nabla^2h(\v)\bar{\v}}&{=-2\alpha\omega_0^2+\w^\top( 6\A\x+2B-2\alpha I_n)\w}\\
&{\leq-2\alpha\omega_0^2+( 6\|\A\|_F+2\|B\|_F)\|\w\|^2-2\alpha\|\w\|^2\leq0,}
\end{aligned}
$$
and the desired result holds.
\qed

\medskip
\noindent{\textbf{Proof of \Cref{them5}.} By \Cref{lema2}, for
$\alpha\ge 3\|\A\|_F+\|B\|_F$ the shifted objective $h$ is concave on the feasible set,
so the homogenisation~\eqref{e22} is exact; the equivalence of \eqref{e22} and its
CSMOP~\eqref{e23}, and the recovery of the optimisers, then follow from the argument of
Theorems~\ref{thm1}--\ref{them2} applied in dimension $n+1$. \qed}

\begin{algorithm}[!htbp]
	\caption{(PAM for the CP-ball problem~\eqref{e19}); identical to \Cref{alg1} with the data of the shifted polynomial $h$.}\label{alg:cpball}
	\begin{algorithmic}[1]
		\STATE Let $\alpha\geq3\|\A\|_F+\|B\|_F$, {tolerance $\varepsilon_{\rm in}>0$, $\|{\bf u}^{(0)}\|=\|\v^{(0)}\|=\|\w^{(0)}\|=1$, proximal weights $\beta_i> 0$}$,\;(i=1,2,3)$.
		\FOR{$\kappa=0,1,2,\ldots$}
		\STATE Update $({\bf u}^{(\kappa+1)},\v^{(\kappa+1)},\w^{(\kappa+1)})$ sequentially via
		\begin{subnumcases}{\label{PAMA2}}
		{\bf u}^{(\kappa+1)}=\arg\min\limits_{\|{\bf u}\|=1}\left[F_h({\bf u},\v^{(\kappa)},\w^{(\kappa)})+{\frac{\beta_1}{2}}\|{\bf u}-{\bf u}^{(\kappa)}\|^2\right],\label{PAMA-4} \\
		\v^{(\kappa+1)}=\arg\min\limits_{\|\v\|=1}\left[F_h({\bf u}^{(\kappa+1)},\v,\w^{(\kappa)})+{\frac{\beta_2}{2}}\|\v-\v^{(\kappa)}\|^2\right], \label{PAMA-5}\\
		\w^{(\kappa+1)}=\arg\min\limits_{\|\w\|=1}\left[F_h({\bf u}^{(\kappa+1)},\v^{(\kappa+1)},\w)+{\frac{\beta_3}{2}}\|\w-\w^{(\kappa)}\|^2\right]. \label{PAMA-6}
		\end{subnumcases}
		\STATE {Set $\v_{\rm best}^{(\kappa+1)}:=\arg\min\left\{h({\bf u}^{(\kappa+1)}),h({\bf v}^{(\kappa+1)}),h({\bf w}^{(\kappa+1)})\right\}$ and write $\v_{\rm best}^{(\kappa+1)}=(\theta^{(\kappa+1)},{\x^{(\kappa+1)}}^{\top})^\top$.}
		\STATE {Stop when the maximum projected block KKT residual is at most $\varepsilon_{\rm in}$, and return $\x^{(\kappa+1)}$.}
		\ENDFOR
	\end{algorithmic}
\end{algorithm}

\begin{remark}[Tensor-free contraction]\label{rem:closed-form-contraction}
The augmented tensor is never assembled. Each block update of \Cref{alg1} uses the
closed-form contraction
\begin{equation}
\label{eq:closed-contraction}
(\mathbf 0,I_n)\,\T\,\tilde\y\,\tilde\z
\;=\;
\tfrac13\,\g \;+\; \tfrac16\,\H(\y+\z) \;+\; \tfrac16\,\Big(\T_0[\y,\z]\Big),
\qquad
\big(\T_0[\y,\z]\big)_i = \sum_{j,l=1}^n (\T_0)_{ijl}\,y_j z_l,
\end{equation}
where the shift contributes $-\tfrac{\alpha}{3}(\y+\z)$. Hence an update costs
$O(\mathrm{nnz}(\H)+\mathrm{nnz}(\T_0))$, which is linear in $n$ for the structured
Chebyshev--Rosenbrock tensor. Equation~\eqref{eq:closed-contraction} is the order-$3$
case of $(\mathbf0,I_n)\A\tilde\x^{p-1}=\tfrac1p\nabla T_{p-1}(\x)$.
\end{remark}

\begin{table}[!htbp]
\centering
\caption{Wall--clock time of PAM ($20$ starts) versus the order--$2$ SDP
relaxation; the rightmost columns show PAM alone at dimensions where the SDP is
intractable.}
\label{tab:cspop-speed}
\begin{tabular}{c|ccc|cccc}
\hline
$n$ & 4 & 8 & 10 & 20 & 40 & 60 & 80 \\
\hline
$t_{\mathrm{PAM}}$ (s) & $0.002$ & $0.004$ & $0.015$ & $0.10$ & $0.81$ & $3.0$ & $7.0$ \\
$t_{\mathrm{SDP}}$ (s) & $0.15$  & $0.88$  & $2.16$  & \multicolumn{4}{c}{intractable} \\
\hline
\end{tabular}
\end{table}

\begin{table}[!htbp]
\centering
\caption{PAM versus other sphere solvers at $n=8$ (global value from the
order--$2$ certificate); times are medians over $20$ random instances from the indefinite/dense class of \Cref{tab:cspop-structures}.}
\label{tab:cspop-rest}
\begin{tabular}{lccc}
\hline
method & reached global (multistart) & reached global (single start) & median time \\
\hline
\textbf{PAM} & $15/15$ & $16/20$ & $\mathbf{4}$\,ms \\
shifted power & $15/15$ & $16/20$ & $18$\,ms \\
Riemannian GD & $15/15$ & $16/20$ & $57$\,ms \\
order--$2$ SDP & (certifies) & --- & $885$\,ms \\
\hline
\end{tabular}
\end{table}

\begin{figure}[!htbp]
\centering
\includegraphics[width=\textwidth]{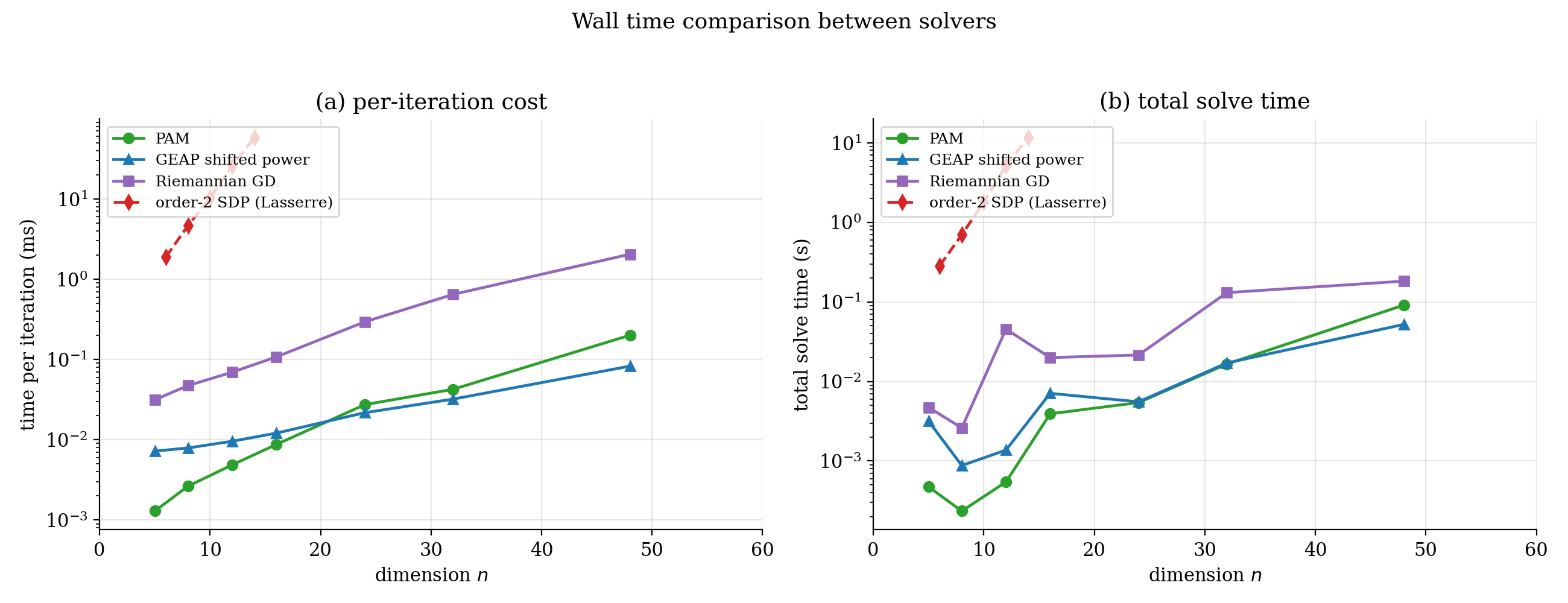}
\caption{Large-scale wall-time comparison on the CSPOP subproblem.
(a) Per-iteration cost: PAM (Numba-JIT) is lowest, about $5\times$ faster per step than
the shifted power method at small $n$, narrowing towards $n\approx 100$; Riemannian GD
is highest because each step runs an Armijo line search with several $O(n^3)$
evaluations. (b) Total solve time over $n$.
dimension.}
\label{fig:cspop-walltime}
\end{figure}

\begin{figure}[!htbp]
\centering
\includegraphics[width=0.8\textwidth]{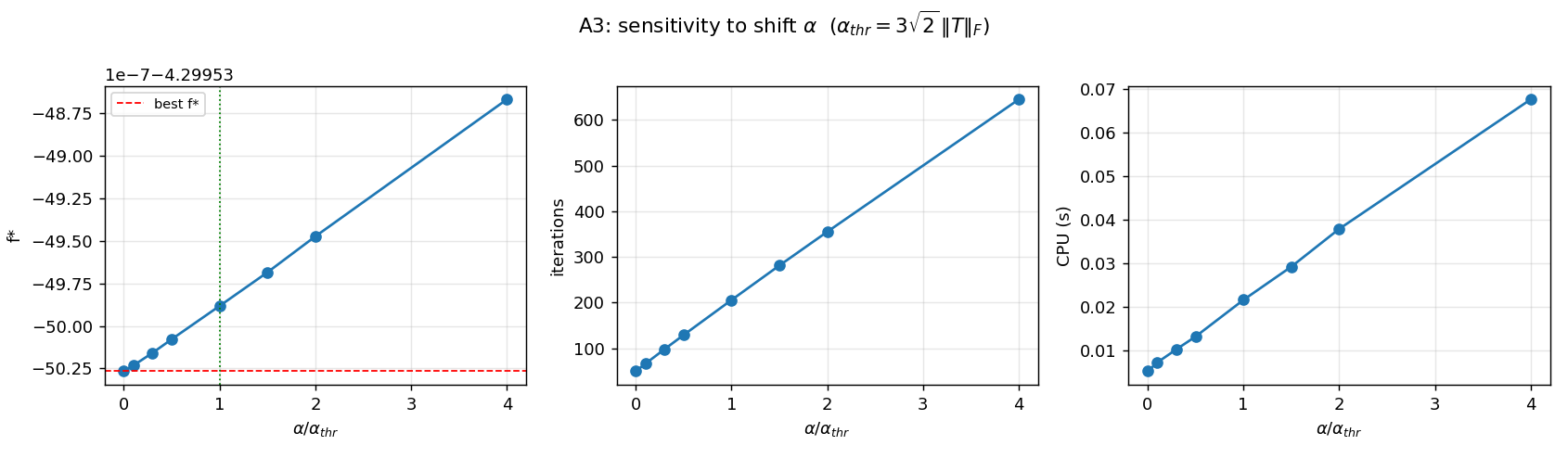}\hfill
\includegraphics[width=0.8\textwidth]{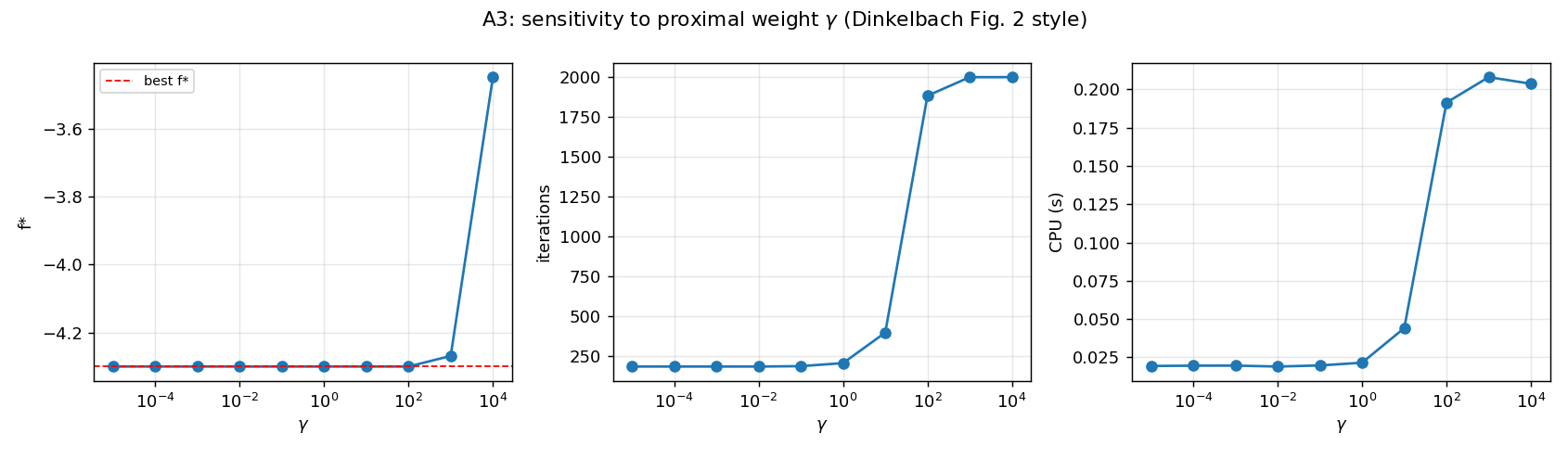}
\caption{Sensitivity of the PAM solver to the shift $\alpha$ (top plot) and the proximal
weights $\gamma_i$ (bottom plot) on the CP-ball problem. The solver is robust across a wide range
of both parameters; we use $\alpha=3\sqrt2\,\|\T\|_F$ and $\gamma_i=1$.}
\label{fig:cspop-alpha-gamma}
\end{figure}

\newpage
\clearpage

\scriptsize
\bibliographystyle{plainnat}
\bibliography{sample}

@article{kolda2014adaptive,
  title   = {An Adaptive Shifted Power Method for Computing Generalized Tensor Eigenpairs},
  author  = {Kolda, Tamara G. and Mayo, Jackson R.},
  journal = {SIAM Journal on Matrix Analysis and Applications},
  volume  = {35},
  number  = {4},
  pages   = {1563--1581},
  year    = {2014},
  doi     = {10.1137/140951758}
}

@book{absil2008optimization,
  title     = {Optimization Algorithms on Matrix Manifolds},
  author    = {Absil, P.-A. and Mahony, R. and Sepulchre, R.},
  year      = {2008},
  publisher = {Princeton University Press},
  address   = {Princeton, NJ}
}

@article{zhu2026sufficiently,
  title={Sufficiently Regularized Nonnegative Quartic Polynomials are Sum-of-Squares},
  author={Zhu, Wenqi and Cartis, Coralia},
  journal={arXiv preprint arXiv:2601.20418},
  year={2026}
}

@article{welzel2025local,
  title={Local Convergence of Adaptively Regularized Tensor Methods},
  author={Welzel, Karl and Liu, Yang and Hauser, Raphael A and Cartis, Coralia},
  journal={arXiv preprint arXiv:2510.25643},
  year={2025}
}

@article{silina2022unregularized,
  title={An unregularized third order Newton method},
  author={Silina, Olha and Zhang, Jeffrey},
  journal={arXiv preprint arXiv:2209.10051},
  year={2022}
}

@article{zhu2023cubic,
  title={Cubic-quartic regularization models for solving polynomial subproblems in third-order tensor methods},
  author={Cartis, Coralia and Zhu, Wenqi},
  journal={Mathematical Programming},
  pages={1--53},
  year={2025},
  publisher={Springer}
}

@inproceedings{zhu2022quartic,
  title={Quartic Polynomial Sub-problem Solutions in Tensor Methods for Nonconvex Optimization},
  author={Zhu, Wenqi and Cartis, Coralia},
  booktitle={NeurIPS 2022 Workshop},
  year={2022}
}

@article{cartis2023second,
  title={Second-order methods for quartically-regularised cubic polynomials, with applications to high-order tensor methods: C. Cartis, W. Zhu},
  author={Cartis, Coralia and Zhu, Wenqi},
  journal={Mathematical Programming},
  pages={1--47},
  year={2025},
  publisher={Springer}
}

@article{lasserre2001global,
  title={Global optimization with polynomials and the problem of moments},
  author={Lasserre, Jean B},
  journal={SIAM Journal on optimization},
  volume={11},
  number={3},
  pages={796--817},
  year={2001},
  publisher={SIAM}
}

@book{conn2000trust,
  title={Trust region methods},
  author={Conn, Andrew R and Gould, Nicholas IM and Toint, Philippe L},
  year={2000},
  publisher={SIAM}
}

@article{chen2025tensor,
  title={Tensor-based Dinkelbach method for computing generalized tensor eigenvalues and its applications},
  author={Chen, Haibin and Zhu, Wenqi and Cartis, Coralia},
  journal={Applied Numerical Mathematics},
  year={2026}
}

@techreport{nesterov2003random,
  title={Random walk in a simplex and quadratic optimization over convex polytopes},
  author={Nesterov, Yurii and others},
  year={2003},
  institution={CORE}
}

@article{luo2010semidefinite,
  title={A semidefinite relaxation scheme for multivariate quartic polynomial optimization with quadratic constraints},
  author={Luo, Zhi-Quan and Zhang, Shuzhong},
  journal={SIAM Journal on Optimization},
  volume={20},
  number={4},
  pages={1716--1736},
  year={2010},
  publisher={SIAM}
}

@article{zhang2012cubic,
  title   = {The Cubic Spherical Optimization Problem},
  author  = {Zhang, Xinzhen and Qi, Liqun and Ye, Yinyu},
  journal = {Mathematics of Operations Research},
  volume  = {37},
  number  = {2},
  pages   = {333--349},
  year    = {2012},
  doi     = {10.1287/moor.1110.0523}
}

@article{nesterov2006cubic,
  title={Cubic regularization of Newton method and its global performance},
  author={Nesterov, Yurii and Polyak, Boris T},
  journal={Mathematical Programming},
  volume={108},
  number={1},
  pages={177--205},
  year={2006},
  publisher={Springer}
}

@article{nesterov2022quartic,
  title={Quartic Regularity},
  author={Nesterov, Yurii},
  journal={arXiv preprint arXiv:2201.04852},
  year={2022}
}

@article{cartis2022evaluation,
  title={Evaluation complexity of algorithms for nonconvex optimization},
  author={Cartis, Coralia and Gould, Nicholas Ian Mark and Toint, Ph L},
  journal={MOS-SIAM Series on Optimization.},
  year={2022}
}

@article{more1981testing,
  title={Testing Unconstrained Optimization Software},
  author={Mor{\'e}, Jorge J and Garbow, Burton S and Hillstrom, Kenneth E},
  journal={ACM Transactions on Mathematical Software (TOMS)},
  volume={7},
  number={1},
  pages={17--41},
  year={1981},
  publisher={ACM New York, NY, USA}
}

@article{cartis2019universal,
  title={Universal regularization methods: varying the power, the smoothness and the accuracy},
  author={Cartis, Coralia and Gould, Nick I and Toint, Philippe L},
  journal={SIAM Journal on Optimization},
  volume={29},
  number={1},
  pages={595--615},
  year={2019},
  publisher={SIAM}
}

@article{8,
  title={Second-order optimality and beyond: characterization and evaluation complexity in nonconvex convexly-constrained optimization},
  author={Cartis, C and Gould, NIM and Toint, Ph L},
  journal={Preprint RAL-P-2016-008, Rutherford Appleton Laboratory, Chilton, Oxfordshire, England},
  year={2016}
}

@article{cartis2020sharp,
  title={Sharp worst-case evaluation complexity bounds for arbitrary-order nonconvex optimization with inexpensive constraints},
  author={Cartis, Coralia and Gould, Nicholas IM and Toint, Philippe L},
  journal={SIAM Journal on Optimization},
  volume={30},
  number={1},
  pages={513--541},
  year={2020},
  publisher={SIAM}
}

@article{cartis2020concise,
  title={A concise second-order complexity analysis for unconstrained optimization using high-order regularized models},
  author={Cartis, Coralia and Gould, Nick IM and Toint, Ph L},
  journal={Optimization Methods and Software},
  volume={35},
  number={2},
  pages={243--256},
  year={2020},
  publisher={Taylor \& Francis}
}

@article{nesterov2021implementable,
  title={Implementable tensor methods in unconstrained convex optimization},
  author={Nesterov, Yurii},
  journal={Mathematical Programming},
  volume={186},
  number={1},
  pages={157--183},
  year={2021},
  publisher={Springer}
}

@article{carmon2020lower,
  title={Lower bounds for finding stationary points I},
  author={Carmon, Yair and Duchi, John C and Hinder, Oliver and Sidford, Aaron},
  journal={Mathematical Programming},
  volume={184},
  number={1},
  pages={71--120},
  year={2020},
  publisher={Springer}
}

@article{cai2026globally,
  title={A Globally Convergent Third-Order Newton Method via Unified Semidefinite Programming Subproblems},
  author={Cai, Yubo and Zhu, Wenqi and Cartis, Coralia and Zardini, Gioele},
  journal={arXiv preprint arXiv:2603.09682},
  year={2026}
}

@article{buchheim2021lower,
  title={Lower bounds for cubic optimization over the sphere},
  author={Buchheim, Christoph and Fampa, Marcia and Sarmiento, Orlando},
  journal={Journal of Optimization Theory and Applications},
  volume={188},
  number={3},
  pages={823--846},
  year={2021},
  publisher={Springer}
}

@article{nesterov2020inexact,
  title={Inexact accelerated high-order proximal-point methods},
  author={Nesterov, Yurii},
  journal={Mathematical Programming},
  pages={1--26},
  year={2021},
  publisher={Springer}
}

@article{nesterov2021inexact,
  title={Inexact high-order proximal-point methods with auxiliary search procedure},
  author={Nesterov, Yurii},
  journal={SIAM Journal on Optimization},
  volume={31},
  number={4},
  pages={2807--2828},
  year={2021},
  publisher={SIAM}
}

@article{carmon2021lower,
  title={Lower bounds for finding stationary points II: first-order methods},
  author={Carmon, Yair and Duchi, John C and Hinder, Oliver and Sidford, Aaron},
  journal={Mathematical Programming},
  volume={185},
  number={1},
  pages={315--355},
  year={2021},
  publisher={Springer}
}

@article{birgin2017worst,
  title={Worst-case evaluation complexity for unconstrained nonlinear optimization using high-order regularized models},
  author={Birgin, Ernesto G and Gardenghi, JL and Mart{\'\i}nez, Jos{\'e} Mario and Santos, Sandra Augusta and Toint, Ph L},
  journal={Mathematical Programming},
  volume={163},
  number={1},
  pages={359--368},
  year={2017},
  publisher={Springer}
}

@article{birgin2018fortran,
  title={Fortran routines for testing unconstrained optimization software with derivatives up to third-order},
  author={Birgin, EG and Gardenghi, JL and Mart{\'\i}nez, JM and Santos, SA},
  year={2018}, 
  journal={.},
}

@article{cartis2007adaptive,
  title={Adaptive cubic overestimation methods for unconstrained optimization},
  journal={ora.ox.ac.uk},
  author={Cartis, Coralia and Gould, Nicholas IM and Toint, Philippe L},
  year={2007},
  
}

@article{chen2022efficient,
  title={An efficient alternating minimization method for fourth degree polynomial optimization},
  author={Chen, Haibin and He, Hongjin and Wang, Yiju and Zhou, Guanglu},
  journal={Journal of Global Optimization},
  volume={82},
  number={1},
  pages={83--103},
  year={2022},
  publisher={Springer}
}

@article{zhu2024global,
  title={Global Convergence of High-Order Regularization Methods with Sums-of-Squares Taylor Models},
  author={Zhu, Wenqi and Cartis, Coralia},
  journal={arXiv preprint arXiv:2404.03035},
  year={2024}
}

@misc{nesterov2008private,
  author       = {Nesterov, Yurii},
  title        = {Private communication},
  year         = {2008},
  note         = {Personal communication}
}

@article{jarre2011nesterov,
  title={On Nesterov’s smooth Chebyshev-Rosenbrock function},
  author={Jarre, Florian},
  journal={rN (x)},
  volume={1},
  pages={4},
  year={2011}
}

@article{gurbuzbalaban2012nesterov,
  title={On Nesterov’s nonsmooth Chebyshev--Rosenbrock functions},
  author={G{\"u}rb{\"u}zbalaban, Mert and Overton, Michael L},
  journal={Nonlinear Analysis: Theory, Methods \& Applications},
  volume={75},
  number={3},
  pages={1282--1289},
  year={2012},
  publisher={Elsevier}
}

@article{gould2016note,
    author = {Gould, Nicholas and Scott, Jennifer},
    doi = {10.1145/2950048},
    journal = {ACM Transactions on Mathematical Software (TOMS)},
    number = {2},
    pages = {15},
    publisher = {ACM},
    title = {{A note on performance profiles for benchmarking software}},
    volume = {43},
    year = {2016}
}

@article{dolan2002benchmarking,
    author = {Dolan, Elizabeth D and Mor{\'e}, Jorge J},
    doi = {10.1007/s101070100263},
    journal = {Mathematical programming},
    number = {2},
    pages = {201--213},
    publisher = {Springer},
    title = {{Benchmarking optimization software with performance profiles}},
    volume = {91},
    year = {2002}
}

@article{nesterov2021superfast,
  title={Superfast second-order methods for unconstrained convex optimization},
  author={Nesterov, Yurii},
  journal={Journal of Optimization Theory and Applications},
  volume={191},
  number={1},
  pages={1--30},
  year={2021},
  publisher={Springer}
}

@article{ahmadi2023higher,
  title={Higher-Order Newton Methods with Polynomial Work per Iteration},
  author={Ahmadi, Amir Ali and Chaudhry, Abraar and Zhang, Jeffrey},
  journal={arXiv preprint arXiv:2311.06374},
  year={2023}
}

@article{curtis2023itrace,
  author  = {Curtis, Frank E. and Wang, Qi},
  title   = {Worst-Case Complexity of {TRACE} with Inexact Subproblem
             Solutions for Nonconvex Smooth Optimization},
  journal = {SIAM Journal on Optimization},
  year    = {2023},
  volume  = {33},
  number  = {3},
  pages   = {2191--2221},
  doi     = {10.1137/22M1492428}
}

@article{cartis2021inexactTR,
  title={Strong evaluation complexity of an inexact trust-region algorithm for arbitrary-order unconstrained nonconvex optimization},
  author={Cartis, Coralia and Gould, Nicholas Ian Mark and Toint, Ph L},
  journal={arXiv preprint arXiv:2011.00854},
  year={2020}
}

@article{curtis2017trace,
  author  = {Curtis, Frank E. and Robinson, Daniel P. and Samadi, Mohammadreza},
  title   = {A trust region algorithm with a worst-case iteration complexity of ${O}(\epsilon^{-3/2})$ for nonconvex optimization},
  journal = {Mathematical Programming},
  volume  = {162},
  number  = {1--2},
  pages   = {1--32},
  year    = {2017}
}

\end{document}